\documentclass[12pt,a4paper]{article}%
\usepackage[utf8]{inputenc}
\usepackage{hyperref}
\usepackage{amsmath}
\usepackage{amsfonts}
\usepackage{amssymb}
\usepackage{xcolor}
\usepackage[space]{grffile}
\usepackage{graphicx}%
\providecommand{\U}[1]{\protect\rule{.1in}{.1in}}
\newtheorem{theorem}{Theorem}

\newtheorem{lemma}[theorem]{Lemma}

\newtheorem{problem}[theorem]{Problem}
\newtheorem{proposition}[theorem]{Proposition}

\newtheorem{observation}[theorem]{Observation}

\newenvironment{proof}[1][Proof]{\noindent\textbf{#1.} }{\ \hfill \rule{0.5em}{0.5em}\bigskip}
\graphicspath{{D:/Dropbox/Riste-Sedlar/GeodesicSubpath/Slike/}}

\begin{document}

\title{Counting geodesic paths in graphs}
\author{Martin Knor$^{1}$, Jelena Sedlar$^{2,4}$, Riste \v{S}krekovski$^{3,4,5}$,
Xiao-Dong Zhang$^{6}$\\{\small $^{1}$ \textit{Slovak University of Technology, Bratislava, Slovakia
}}\\[0.1cm] {\small $^{2}$ \textit{University of Split, FGAG, Split, Croatia }}\\[0.1cm] {\small $^{3}$ \textit{University of Ljubljana, FMF, Ljubljana,
Slovenia }}\\[0.1cm] {\small $^{4}$ \textit{Faculty of Information Studies, Novo Mesto,
Slovenia }}\\[0.1cm] {\small $^{5}$ \textit{Rudolfovo - Science and Technology Centre, Novo
Mesto}\textit{, Slovenia }}\\[0.1cm] {\small $^{6}$ \textit{School of Mathematical Sciences, MOE-LSC and
SHL-MAC, }}\\{\small \textit{Shanghai Jiao Tong University, Shanghai, PR China }}}
\maketitle

\begin{abstract}
A geodesic is a shortest path which connects a pair of vertices of a graph
$G$. In this paper we define the geodesic subpath number $\mathrm{gpn}(G)$ of
a graph $G$ as the number of geodesics in $G$. The number of subtrees and
subpaths are already studied in literature, but they are both large
quantities. Hence, the geodesic subpath number which is related to these
quantities but smaller than both, seems worthy of investigation. We first
consider extremal graphs with respect to the geodesic subpath number among all
connected graphs on $n$ vertices. This number is minimized by the so called
geodetic graphs, i.e. graphs in which each pair of vertices is connected by
precisely one geodesic. As for the graphs which maximize the geodesic subpath
number, we provide an upper bound on $\mathrm{gpn}(G)$ in terms of $n$ and we
further consider several graph families which might have a large
$\mathrm{gpn}(G)$. Yet, their value of $\mathrm{gpn}(G)$ still does not attain
the established bound, so narrowing the gap remains as an open problem. We
also consider the class of cactus graphs on $n$ vertices and $k$ cycles and
among them characterize extremal graphs with respect to this new invariant.

\end{abstract}



\section{Introduction}

In various problems including graphs, certain types of subgraphs play an
important role. So, it is of interest to count the number of such subgraphs.
For example, the number of spanning subtrees of a graph $G$ is extensively
studied, such as when $G$ is a complete multipartite graph \cite{YanSpanning},
a labeled molecular graph \cite{BrownSpanning}, a circulant graph
\cite{AtajanSpanning}, a quasi-threshold graph \cite{BapatSpanning}, a
multi-star related graph \cite{NikopoulosSpanning}, algorithms for finding
this number are also developed \cite{NikolopoulosAlgoritamSpanning}, etc.

It is of interest also to count subtrees of a graph $G$ without the additional
requirement that the subtree is spanning. So, the number of non-empty subtrees
of a graph $G,$ denoted by $N(G)$, also received much attention. It was first
considered for the case when $G$ is a tree \cite{Szekely2005}. The case when
$G$ belongs to various subfamilies of trees has also been studied, for example
binary trees \cite{Szekely2007}, trees with given maximum degree
\cite{Kirk2008}, trees with given degree sequence \cite{Zhang2013}. As of
lately, various broader families of graphs in relation to the subtree number
have also been researched, including graphs with given number of cut edges
\cite{Xu2021} and cacti and block graphs \cite{Cacti2022}. An interesting
result regarding the inverse problem for the subtree number has also been
obtained \cite{Czabarka2008}.

A subpath is a special kind of a subtree, and it is an important substructure
in many problems modeled by graphs. Also, the number of subtrees is usually
very large, so the number of subpaths which is smaller should be more
manageable. Nevertheless, it is $\#P$-hard to count subpaths, so the problem
is not trivial \cite{Yamamoto2017}. Motivated by all this, in our previous
papers we studied the subpath number of a graph and we derived its value for
some classes of graphs such as bipartite graphs, cycle chains \cite{Knor2025}
and cactus graphs \cite{Knor2025b}.

When considering subpaths, the shortest subpaths are of the greatest interest
in many problems. There are many papers which deal with the problem of
counting the shortest paths, mainly between two given vertices and from the
algorithmic point of view. To mention just a few recent ones, the researched
problems are counting shortest paths between two nodes in large road-network
graphs \cite{QiuShortest}, maintaining shortest-path counting indexes under
dynamic graph changes \cite{FengShortest}, addressing the bottleneck of
existing shortest-path \cite{PengShortest}.

Given all the context mentioned above, the geodesic subpath number of a graph
is defined as the number of all shortest paths in $G.$ A shortest path between
a pair of vertices in $G$ is called a \emph{geodesic}, hence the name. To be
more precise, for a pair of vertices $u$ and $v$ of a graph $G$, first the
geodesic subpath number $\mathrm{gpn}_{G}(u,v)$ is defined to be the number of
shortest paths in $G$ connecting $u$ and $v.$ Also, we define $\mathrm{gpn}%
_{G}(u)=%
{\displaystyle\sum_{v\in V(G)}}
\mathrm{gpn}_{G}(u,v).$ We will omit the subscript $G$ when the graph is clear
from the context. The \emph{geodesic subpath number} $\mathrm{gpn}(G)$ is
defined by%
\[
\mathrm{gpn}(G)=\sum_{\{u,v\}\subseteq V(G)}\mathrm{gpn}_{G}(u,v)+n,
\]
since there are $n$ geodesics of zero length. This quantity is related to the
number of subtrees and the number of subpaths, since a geodesic is both a
subtree and a subpath of a graph $G,$ but not every subtree or a subpath is a
geodesic. Hence, the geodesic subpath number is smaller than both the subtree
number and the subpath number of a graph. The concept of counting geodesics
has already been considered in the context of Leech labeling \cite{Indijka}.

There are many other graph theoretical concepts which involve shortest paths,
mainly their length, such as various metric dimensions \cite{Kuziak, Peterin}
and also the Wiener index \cite{risteSurvey1, risteSurvey2}, so it would be
interesting to investigate the possible relation between the geodesic subpath
number and these quantities.

In this paper, all graphs are tacitly assumed to be simple and connected, and
we study the geodesic subpath number of such graphs. Our main interest are the
graphs which minimize and the graphs which maximize the geodesic subpath
number. It is easily seen that the minimal graphs are the so called geodetic
graphs, i.e. graphs in which every pair of vertices is connected by precisely
one shortest path.

The question of maximal graphs does not seem to be so easy. We provide an
upper bound on the geodesic subpath number in terms of the number of vertices.
Next, we provide a formula for the exact value of the geodesic subpath number
for several graph classes which are likely to have a large geodesic subpath
number, such as the family of graphs $G_{k,t}$ obtained as a sequential join
of multiple copies of empty graphs, the hypercubes and the grids. Using these
formulas, we establish that among these graphs $G_{3,n/3}$ has the largest
value of the geodesic subpath number. Since $\mathrm{gpn}(G_{3,n/3})$ still
does not attain the established upper bound, the problem of narrowing the gap
by either lowering the upper bound or finding graphs $G$ with larger value of
$\mathrm{gpn}(G_{3,n/3})$ remains open. We also consider the class of cactus
graphs on $n$ vertices having $k$ cycles and characterize extremal graphs in
this class.

\section{Extremal graphs}

A shortest path which connects a pair of vertices $u$ and $v$ is called
$(u,v)$\emph{-geodesic}. The \emph{distance} between vertices $u$ and $v$ of
$G,$ denoted by $\mathrm{dist}_{G}(u,v),$ is the length of a $(u,v)$-geodesic.
A \emph{geodetic graph} is a graph $G$ where every pair of vertices $u$ and
$v$ of $G$ has a unique $(u,v)$-geodesic connecting them.

We want to establish which graphs minimize the geodesic subpath number and
which ones maximize it among all connected graphs on $n$ vertices. In a
connected graph every pair of vertices $u$ and $v$ is connected by at least
one path, so $\mathrm{gpn}(u,v)\geq1$ which implies $\mathrm{gpn}(G)\geq
\binom{n}{2}+n=\binom{n+1}{2}.$ Here, the equality holds if and only if $G$ is
a geodetic graph, i.e. if $\mathrm{gpn}(u,v)=1$ for every $u,v\in V(G).$ We
state this formally in the following observation.

\begin{observation}
\label{Observation_geodetic}Let $G$ be a graph on $n$ vertices. Then,
$\mathrm{gpn}(G)\geq\binom{n+1}{2}$ with equality if and only if $G$ is a
geodetic graph.
\end{observation}

The notion of geodetic graph was first introduced in \cite{Ore1962} and it was
extensively studied since, let us mention just a few recent papers
\cite{Cornelsen2022, Elder2021, Gorovoy}. The class of geodetic graphs
contains trees, the Petersen graph, the Hoffman--Singleton graph, strongly
regular graphs where every two non-adjacent vertices have one common neighbor,
complete graphs, Gallai trees, cactus graphs in which all cycles are of odd length.

\bigskip

Hence, we have resolved the question of which graphs minimize the geodesic
subpath number. Determining which graphs maximize the geodesic subpath number
appears to be considerably more challenging. Let us first establish the upper
bound on $\mathrm{gpn}(G)$ for a general graph on $n$ vertices.

\begin{theorem}
\label{thm:bound} Let $G$ be a graph on $n\geq1$ vertices. Then
\[
\mathrm{gpn}(G)\leq\frac{n}{4}(23\cdot3^{\frac{n-6}{3}}+1).
\]

\end{theorem}

\begin{proof}
For $n\leq6,$ it is easily verified that the bound holds. In particular, the
exstremal graphs for the geodesic subpath number are $K_{1},$ $K_{2},$
$\{K_{3},K_{1,2}\},$ $K_{2,2},$ $K_{2,3},$ $K_{3,3}$ for $n=1,\ldots,6,$
respectively. Hence, let us assume that $n\geq7$ and let $u\in V(G)$. We will
establish the upper bound on $\mathrm{gpn}(u),$ i.e. the number of geodesics
starting at $u.$ Given a vertex $u,$ we define
\[
A_{i}(u)=\{w\in V(G):\mathrm{dist}(u,w)=i\}.
\]
For the sake of simplicity, we will write $A_{i}$ instead of $A_{i}(u)$ when
this does not lead to confusion. Denote $a_{i}=\left\vert A_{i}\right\vert $
and notice that $A_{0}=\{u\}$. Also, for $v\in A_{i}(u),$ where $i\geq1,$
every $(u,v)$-geodesic contains precisely one vertex from $A_{j}$ for every
$j\in\{1,\ldots,i-1\}$. Therefore, the $\mathrm{gpn}(u,v)$ attains its maximum
if every choice of vertices $v_{1},\dots,v_{i-1}$, where $v_{j}\in A_{j}$,
gives a path. Hence, $a_{1}\cdot a_{2}\cdots a_{i-1}$ is the upper bound on
$\mathrm{gpn}(u,v).$ Denote by $t=\mathrm{\max}\{\mathrm{dist}(u,v):v\in
V(G)\}$ the eccentricity of $u.$ Let us assume that $G$ is a graph on $n$
vertices and $u$ a vertex of $G$ such that $\mathrm{gpn}(u)$ attains the
maximum possible value. In the following claims, we give some structural
properties of $A_{i}$'s for such a vertex $u.$

\bigskip

\noindent\textbf{Claim A.} \emph{It holds that }$a_{i}\leq3$\emph{ for every
}$i.$

\smallskip

\noindent Assume first that there is an $A_{i}$ with $a_{i}\geq5.$ We split
the set $A_{i}$ into two smaller ones $A_{i_{1}}$ and $A_{i_{2}},$ with
$\left\vert A_{i_{1}}\right\vert =2$ and $|A_{i_{2}}|=a_{i}-2$, so that
$\mathrm{dist}(u,v)$ increases by $1$ for every vertex $v\in A_{i_{2}}$ and
for all vertices of $v\in A_{j}$ for $j>i.$ Notice that $\mathrm{gpn}(u,v)$
remains the same for every $v\in A_{j},$ where $j\leq i-1$ and for $v\in
A_{i_{1}}.$ For $v\in A_{i_{2}},$ the $\mathrm{gpn}(u,v)$ increases by
multiple $2.$ For $v\in A_{j},$ where $j>i,$ a factor $a_{i}$ is replaced by
$2(a_{i}-2)$ in an upper bound $a_{1}\cdot a_{2}\cdots a_{i-1}$ for
$\mathrm{gpn}(u,v).$ Since $2(a_{i}-2)$ is larger than $a_{i}$ for $a_{i}%
\geq5,$ we obtain the claim.

Assume next that there is an $A_{i}$ with $a_{i}=4.$ In this case, we can
split $A_{i}$ into two sets $A_{i_{1}}$ and $A_{i_{2}}$, both of size $2$.
Then, for all vertices $v\in A_{i_{1}}$ and all vertices $v\in A_{j},$ where
$j<i,$ the value of $\mathrm{gpn}(u,v)$ remains the same. Also, for every
vertex $v\in A_{j},$ where $j>i,$ the factor $a_{i}=4$ is replaced by
$2(a_{i}-2)=4,$ so the value of $\mathrm{gpn}(u,v)$ again remains the same.
Yet, for vertices $v\in A_{i_{2}},$ the value of $\mathrm{gpn}(u,v)$ increases
by the multiple of $2,$ which yields the claim.

\bigskip

\noindent\textbf{Claim B.} \emph{If }$i<k$\emph{, then }$a_{i}\geq a_{k}%
$\emph{.}

\smallskip

\noindent Assume to the contrary, that $i<k$ and $a_{i}<a_{k}.$ Then, there
must exist two consecutive $A_{i}$ and $A_{i+1}$ such that $a_{i}<a_{i+1}.$
Denote $p=a_{1}\cdots a_{i-1}.$ We may switch $A_{i}$ and $A_{i+1},$ and we
show that this results in the increase of $\mathrm{gpn}(u).$ By Claim A, we
may assume $a_{j}\leq3$ for every $1\leq j\leq t,$ so there are three possibilities.

\begin{itemize}
\item If $a_{i}=1$ and $a_{i+1}=2$ then for $v\in A_{i+1},$ the value of
$\mathrm{gpn}(u,v)$ remains the same after switching, while for $v\in A_{i}$
the value of $\mathrm{gpn}(u,v)$ increases by the factor of $2$. For all other
$v$ the value of $\mathrm{gpn}(u,v)$ remains the same. Hence, $\mathrm{gpn}%
(u)$ increases, a contradiction.

\item If $a_{i}=1$ and $a_{i+1}=3,$ the same argument as above applies, only
the increase of $\mathrm{gpn}(u,v)$ for $v\in A_{i}$ is by the factor of $3$
instead of $2.$

\item If $a_{i}=2$ and $a_{i+1}=3,$ then for each $v\not \in A_{i}\cup
A_{i+1}$ remains the same, similarly as in previous cases. Further, for each
vertex $v\in A_{i}$ the value of $\mathrm{gpn}(u,v)$ increases from $p$ to
$3p,$ while for each vertex $v\in A_{i+1}$ the value of $\mathrm{gpn}(u,v)$
decreases from $2p$ to $p.$ This means that the change of $\mathrm{gpn}(u)$ is
equal to
\[
a_{i}(3p-p)+a_{i+1}(p-2p)=2(3p-p)+3(p-2p)=p>0.
\]
so the value of $\mathrm{gpn}(u)$ increases, again a contradiction.
\end{itemize}

\noindent We have established the claim in all three possible cases, so we are done.

\bigskip

Due to Claim B, we know that the vector $(a_{1},\ldots,a_{t})$ is of the form
$(3^{a},2^{b},1^{c}),$ where $3^{a}$ denotes the $a$ consecutive coordinates
with value $3,$ and similarly holds for $2^{b}$ and $1^{c}.$ Obviously, it
must hold that $n=1+3a+2b+c.$ Let us show some claims regarding $b$ and $c.$

\bigskip

\noindent\textbf{Claim C.} \emph{It holds that }$c\leq2.$

\smallskip

\noindent Assume to the contrary that $c>2.$ If $a_{t-2}=a_{t-1}=1,$ we can
join $A_{t-2}$ and $A_{t-1}$ into a single set of the size $2.$ Then the value
of $\mathrm{gpn}(u,v)$ remains the same for every vertex except for $v\in
A_{t}$ for which it increases by the factor of $2.$ Hence, $\mathrm{gpn}(u)$
is not the maximum possible, a contradiction.

\bigskip

\noindent\textbf{Claim D.} \emph{It holds that }$b\leq3.$ \emph{Moreover,
}$b=3$\emph{ if and only if }$n\equiv2(\operatorname{mod}3)$.

\smallskip

\noindent Let $a_{i}=a_{i+1}=a_{i+2}=2.$ Denote $N_{i}=\sum_{j=i}%
^{t}\left\vert A_{i}\right\vert $ and $n_{i}=\left\vert N_{i}\right\vert $. It
is obvious that $n_{i}\geq6.$

If $n_{i}\geq9,$ we can join three consecutive sets $A_{i}$, $A_{i+1}$ and
$A_{i+2}$ of size two into two consecutive sets $A_{i_{1}}$ and $A_{i_{2}}$ of
size three. Notice that for each vertex $v\in A_{j},$ where $j<i$, the value
of $\mathrm{gpn}(u,v)$ remains the same. Also here denote $p=a_{1}\cdots a_{i-1}$.  The total change of $\mathrm{gpn}(u)$
for vertices $v$ from $A_{i}\cup A_{i+1}\cup A_{i+2}$ is
\[
(3\cdot p+3\cdot3p)-(2\cdot p+2\cdot2p+2\cdot4p)=-2p<0.
\]
But, for all vertices $v\in A_{j}$, where $j>i+2,$ the three factors
$a_{i}\cdot a_{i+1}\cdot a_{i+2}=8$ are replaced by $\left\vert A_{i_{1}%
}\right\vert \cdot\left\vert A_{i_{2}}\right\vert =9,$ so their contribution
to $\mathrm{gpn}(u)$ increases by at least $p$. Since $n_{i}\geq9,$ there are
at least three such vertices, so the total change of $\mathrm{gpn}(u)$ is
positive, a contradiction with $\mathrm{gpn}(u)$ being maximum possible.

If $n_{i}=8,$ then the only vectors $(a_{i},a_{i+1},\ldots,a_{t})$ with
$b\geq3$ are $(2,2,2,2)$ and $(2,2,2,1,1)$. In both these cases the
contribution of vertices $v\in N_{i}$ to $\mathrm{gpn}(u)$ is $30p$. However,
there are two additional possibilities for this vector, namely $(3,3,2)$ and
$(3,2,2,1),$ with contributions of vertices $v\in N_{i}$ to $\mathrm{gpn}(u)$
being $30p$ and~$33p,$ respectively. Since the contribution of vertices
$v\notin N_{i}$ to $\mathrm{gpn}(u)$ is the same in all these cases, the value
of $\mathrm{gpn}(u)$ is maximal in the case of $(3,2,2,1),$ which proves the claim.

If $n_{i}=7,$ then the only vector $(a_{i},a_{i+1},\ldots,a_{t})$ with
$b\geq3$ is $(2,2,2,1).$ In this case, the contribution of vertices $v\in
N_{i}$ to $\mathrm{gpn}(u)$ is $22p.$ Notice that here we have $n-1\equiv
1\pmod{3}$ and $c=1,$ so we will show that for this vector $\mathrm{gpn}(u)$
is maximal. To see this, observe that there are three additional possibilities
for this vector, namely $(3,3,1),$ $(3,2,2),$ and $(3,2,1,1).$ The
contribution of the vertices of $N_{i}$ to $\mathrm{gpn}(u)$ in each of these
three cases is $21p$. Hence, the value of $\mathrm{gpn}(u)$ is maximal in the
case of $(2,2,2,1)$ as claimed.

If $n_{i}=6,$ then the only vector $(a_{i},a_{i+1},\ldots,a_{t})$ with
$b\geq3$ is $(2,2,2).$ In the case of this vector, the contribution of
vertices from $N_{i}$ to $\mathrm{gpn}(u)$ is $12p.$ Let us show that
$\mathrm{gpn}(u)$ is not maximal for this vector. We do that by considering
the three remaining possibilities which are $(3,3),$ $(3,2,1),$ and
$(2,2,1,1).$ The contribution of vertices of $N_{i}$ to $\mathrm{gpn}(u)$ in
the four mentioned cases is $12p,$ $15p,$ and $14p,$ respectively. Hence, the
value of $\mathrm{gpn}(u)$ is maximum in the case $(3,2,1),$ which proves the claim.

Thus, we have now finished the consideration of all the possible cases, which implies that 
$b\le 3$. Since $n\ge 7$, it also implies that $b=3$ if and only if $n-1\equiv1(\operatorname{mod}3)$. This establishes  Claim D.

\bigskip

We are now in a position to establish the upper bound on $\mathrm{gpn}(u).$
Assume $n=1+3q+r$ for $0\leq r\leq2.$ Recall that in Claim D we established
the vector $(a,b,c)$ which gives maximum possible $\mathrm{gpn}(u)$ for
$n_{i}$ equal to $8,$ $7$ and $6.$ These cases correspond to $r$ being equal
to $2,$ $1$ and $0.$ Hence, Claim D implies that $\mathrm{gpn}(u)$ is maximum
possible if and only if%
\[
(a,b,c)=\left\{
\begin{array}
[c]{cc}%
((n-4)/3,1,1) & \text{for }r=0,\\
((n-8)/3,3,1) & \text{for }r=1,\\
((n-6)/3,2,1) & \text{for }r=2.
\end{array}
\right.
\]
Since
\begin{align*}
\mathrm{gpn}(u)  &  =1+\sum_{i=1}^{a}3\cdot3^{i-1}+3^{a}\sum_{i=1}^{b}%
2\cdot2^{i-1}+c\cdot3^{a}\cdot2^{b}\\
&  =2^{b+1}3^{a}-\frac{1}{2}3^{a}+2^{b}3^{a}c-\frac{1}{2},
\end{align*}
plugging the above values of $(a,b,c)$ gives the following upper bounds for
these cases%
\[
\mathrm{gpn}(u)\leq\left\{
\begin{array}
[c]{cc}%
\frac{1}{2}(11\cdot3^{\frac{n-4}{3}}-1) & \text{for }r=0,\\
\frac{1}{2}(47\cdot3^{\frac{n-8}{3}}-1) & \text{for }r=1,\\
\frac{1}{2}(23\cdot3^{\frac{n-6}{3}}-1) & \text{for }r=2.
\end{array}
\right.
\]
It is easily verified that the largest of the three bounds is the one for
$r=2$. Now, from $\mathrm{gpn}(G)=\frac{1}{2}\sum_{u\in V(G)}\mathrm{gpn}%
(u)+\frac{n}{2}$ follows the claim of the theorem and we are finished.
\end{proof}

Next, we consider several classes of graphs which might have a large geodesic
subpath number. The first such class consists of graphs obtained as a
sequential join of multiple copies of empty graphs. More formally, by $D_{k}$
we denote a discrete graph on $k$ vertices, i.e., the complement of $K_{k}$.
Let $G_{k,t}$ be a sequential join of $t$ copies of $D_{k}$ in a way that
$i$-th copy of $D_{k}$ is joined to $(i+1)$-th copy for every $i\in
\{1,\ldots,t-1\}$. In the next theorem we establish the formula for the exact
value of the geodesic subpath number for graphs $G_{k,t}.$

\begin{proposition}
\label{thm:sjoin}Let $k,t\geq2$ be integers. Then
\[
\mathrm{gpn}(G_{k,t})=\frac{1}{k-1}\left(  \frac{k^{t+2}-k^{3}}{k-1}%
+k^{3}(k-2)(t-1)\right)  +kt.
\]

\end{proposition}

\begin{proof}
For $i\in\{1,\ldots,t\}$, denote by $D_{k}^{i}$ the $i$-th copy of $D_{k}$
within $G_{k,t}.$ Here, we assume that copies of $D_{k}^{i}$ are denoted
sequentially, i.e. that contracting all vertices of each copy of $D_{k}$ into
a single vertex results to a path $H=h_{1}h_{2}\cdots h_{t}$ where $D_{k}^{i}$
is contracted into $h_{i}.$

Let us first consider the contribution of a pair of vertices $u\in D_{k}^{i}$
and $v\in D_{k}^{j}$ for $i\not =j.$ Denote $d=j-i$ and notice that
$\mathrm{dist}_{H}(h_{i},h_{j})=d.$ Since $i,j\in\{1,\ldots,t\},$ it follows
that $d\in\{1,\ldots,t-1\}.$ For a given $d,$ indices $i$ and $j$ can be
chosen in $t-d$ ways. Also, for a given $i$ and $j,$ a pair $u$ and $v$ can be
chosen in $k^{2}$ ways. Finally, for a chosen $u$ and $v$ it holds that
$\mathrm{gpn}(u,v)=k^{d-1}.$ Hence, the contribution of pairs $u\in D_{k}^{i}$
and $v\in D_{k}^{j}$, where $i\not =j$, to $\mathrm{gpn}(G)$ is
\[
\sum_{d=1}^{t-1}(t-d)\cdot k^{2}\cdot k^{d-1}=\frac{1}{k-1}\left(
\frac{k^{t+2}-k^{3}}{k-1}-(t-1)k^{2}\right)  .
\]
It remains to consider pairs $u$ and $v$ which belong to the same copy of
$D_{k}.$ So, let $u,v\in D_{k}^{i}$ be such a pair of vertices in $G_{k,t}.$
Notice that $i$ can be chosen in $t$ different ways, and for a given $i$ there
are $\binom{k}{2}$ pairs of vertices $u,v\in D_{k}^{i}.$ Finally, for a given
pair of vertices $u,v\in D_{k}^{i}$ it holds that $\mathrm{gpn}(u,v)=2k$ if
$i\in\{2,\ldots,t-1\},$ and $\mathrm{gpn}(u,v)=k$ if $i\in\{1,t\}.$ We
conclude that the contribution to $\mathrm{gpn}(G)$ of all pairs which belong
to the same copy of $D_{k}$ is
\[
(t-2)\binom{k}{2}\cdot2k+2\binom{k}{2}\cdot k=k^{2}\left(  k-1\right)  \left(
t-1\right)  .
\]
Summing the contribution of pairs which belong to distinct copies of $D_{k}$
and the contribution of pairs which belong to the same copy of $D_{k}$,
together with the $kt$ geodesics of length $0$, yields the claim.
\end{proof}

The formula for the geodesic subpath number in the above result depends on $k$
and $t.$ Since we seek graphs on $n$ vertices which have large geodesic
subpath number and consequently could be good candidates for the maximal
graphs, we wish to derive the formula which depends on the number of vertices
$n$ in $G_{k,t}$. For that purpose, let $n$ be the number of vertices of
$G_{k,t}$. Then $n=kt$ and $t=n/k$. Hence, when $t$ increases we can
approximate with
\[
\mathrm{gpn}(G_{k,t})\sim\left(  \frac{k}{k-1}\right)  ^{2}\cdot k^{n/k}.
\]
If $n$ is small then $(\frac{2}{1})^{2}\cdot2^{n/2}>(\frac{3}{2})^{2}%
\cdot3^{n/3}$. But, from $n\geq30$ we have $(\frac{2}{1})^{2}\cdot
2^{n/2}<(\frac{3}{2})^{2}\cdot3^{n/3}$ and we get smaller values also if
$k\geq4$. Hence, the maximum number of geodesics is obtained when $G$ is the
sequential join of $n/3$ copies of $D_{3}$, in which case
\[
\mathrm{gpn}(G_{3,n/3})\sim\left(  \frac{3}{2}\right)  ^{2}\cdot3^{n/3}.
\]
Note that this is just a linear factor away from the upper bound of Theorem
\ref{thm:bound}.

Next, we determine the exact geodesic subpath number of hypercubes, which
constitute another class of graphs that may attain large, and possibly
extremal, values of this parameter. Let $Q_{r}$ be the $r$-dimensional
hypercube. Then $Q_{r}$ has $n=2^{r}$ vertices. The next theorem gives a
formula for the exact value of the geodesic subpath number for $Q_{r}$ in
terms of $r.$

\begin{proposition}
\label{thm:cube} For an $r$-dimensional hypercube we have
\[
\mathrm{gpn}(Q_{r})=2^{r-1}\cdot r!\sum_{j=0}^{r-1}\frac{1}{j!}+2^{r}%
\sim2^{r-1}\cdot r!\cdot e+2^{r}.
\]

\end{proposition}

\begin{proof}
Let $u,v\in V(Q_{r})$ such that $u\not =v.$ Denote $\mathrm{dist}_{Q_{r}%
}(u,v)=d.$ Notice that $d$ can take its values from the set $\{1,\ldots,r\}.$
We first want to establish the number of pairs $u$ and $v$ at distance $d.$
Notice that $u$ can be chosen in $2^{r}$ distinct ways. Vertices at distance
$d$ from $u$ have precisely $d$ coordinates distinct than those in $u$ and the
remaining $r-d$ coordinates are the same. Hence, there are as many vertices
$v$ at distance $d$ from $u$ as there are distinct choices of $d$ coordinates
in $u,$ i.e. this number is $\binom{r}{d}.$ Notice that we counted each pair
$u,v\in V(Q_{r})$ at distance $d$ twice, hence there are $\frac{1}{2}\binom
{r}{d}2^{r}$ such pairs.

Next, we wish to establish the number of geodesics between a pair of vertices
$u,v\in V(Q_{r})$ with $\mathrm{dist}_{Q_{r}}(u,v)=d\geq1.$ Recall that $u$
and $v$ differ in precisely $d$ coordinates out of $r$ coordinates in total.
In every geodesic $P$ connecting $u$ and $v,$ precisely one coordinate changes
between two neighbors of $P.$ Hence, by choosing the order of changing the $d$
distinct coordinates from the value of $u$ to the value of $v,$ we have chosen
one geodesic. This implies that $\mathrm{gpn}(u,v)=d!.$

Recall that we counted only the pairs of distinct vertices $u$ and $v,$ i.e.
geodesics of length at least one. The number $2^{r}$ of geodesics of length
zero also has to be added into the sum, so we conclude that
\begin{align*}
\mathrm{gpn}(Q_{r})  &  =2^{r}+\sum_{d=1}^{r}\frac{1}{2}\binom{r}{d}2^{r}\cdot
d!=2^{r}+\sum_{d=1}^{r}2^{r-1}\frac{r!}{(r-d)!}\\
&  \sim2^{r-1}\cdot r!\cdot e+2^{r}%
\end{align*}
and the claim is established.
\end{proof}

We wish to compare the geodesic subpath number of the hypercube $Q_{r}$ with
that of $G_{k,t}.$ For that purpose, we first express the geodesic subpath
number of the hypercube $Q_{r}$ in terms of $n.$ Notice that $r=\log_{2}(n),$
so we get%
\[
\mathrm{gpn}(Q_{r})\sim\frac{ne}{2}\left(  \frac{r}{e}\right)  ^{r}=\frac
{ne}{2}\cdot3^{\log_{2}(n)(\log_{3}(\log_{2}(n))-\log_{3}(e))}\ll3^{n/3}.
\]
It is now obvious that the geodesic subpath number of the hypercube $Q_{r}$ is
asymptotically much smaller than the geodesic subpath number of $G_{k,t}.$

Next, we study grids and determine their geodesic subpath number, which
appears to be large and as such a candidate for the extremal graph. Let
$R_{r,s}$ denote the Cartesian product of paths $P_{r}$ and $P_{s},$ i.e.
$R_{r,s}=P_{r}\square P_{s}.$ Notice that the grid $R_{r,s}$ has $n=rs$
vertices. Edges of $R_{r,s}$ which belong to copies of $P_{r}$ (resp. $P_{s}$)
are called \emph{horizontal} (resp. \emph{vertical}). In the next theorem, we
establish the formula for the exact value of the geodesic subpath number of
the grid $R_{r,s}$ in terms of $r$ and $s.$

\begin{proposition}
\label{thm:grid} For the grid $R_{r,s}$ we have
\[
\mathrm{gpn}(R_{r,s})=2\binom{r+s+2}{r+1}-2r-2s-4-\frac{rs}{2}%
\big(r+s+4\big).
\]

\end{proposition}

\begin{proof}
Let us consider a pair of vertices $u,v\in V(R_{r,s}).$ Let $u=(a,b)$ for
$1\leq a\leq r$ and $1\leq b\leq s$. For a given $u,$ we define $K_{i}(u)$ for
$i\in\{1,\ldots,4\}$ to be vertices of the four quadrants determined by $u,$
i.e.
\begin{align*}
K_{1}(u)  &  =\{(x,y):1\leq x\leq a\text{ and }1\leq y\leq b\},\\
K_{2}(u)  &  =\{(x,y):a\leq x\leq r\text{ and }1\leq y\leq b\},\\
K_{3}(u)  &  =\{(x,y):a\leq x\leq r\text{ and }b\leq y\leq s\},\\
K_{4}(u)  &  =\{(x,y):1\leq x\leq a\text{ and }b\leq y\leq s\}.
\end{align*}

Let us first consider the case when $v=(x,y)\in K_{1}(u).$ Traversing a
geodesic from $(a,b)$ to $(x,y)$, we have to use $a-x$ horizontal edges and
$b-y$ vertical edges. Since we can choose the edges in the above mentioned
copies in any order, we have $\mathrm{gpn}((a,b),(x,y))=\binom{a-x+b-y}{a-x}.$
Hence, the contribution to $\mathrm{gpn}(R_{r,s})$ of the pairs of vertices
from this case equals%
\begin{align*}
\mathrm{gpn}_{R_{r,s}}(u,K_{1}(u))  &  =\sum_{y=1}^{b}\sum_{x=1}^{a}%
\binom{a-x+b-y}{a-x}=\sum_{y=1}^{b}\binom{a+b-y}{a-1}\\
&  =\binom{a+b}{a}-\binom{a-1}{a-1}=\binom{a+b}{a}-1.
\end{align*}

Now, for any $u=(a,b)\in V(R_{r,s})$ we define $f_{2}(u)=(r+1-a,b),$
$f_{3}(u)=(r+1-a,s+1-b)$ and $f_{4}(u)=(a,s+1-b).$ Notice that
\begin{align*}
\mathrm{gpn}_{R_{r,s}}(u,K_{1}(u))  &  =\mathrm{gpn}_{R_{r,s}}(f_{2}%
(u),K_{2}(f_{2}(u)))=\mathrm{gpn}_{R_{r,s}}(f_{3}(u),K_{3}(f_{3}(u)))\\
&  =\mathrm{gpn}_{R_{r,s}}(f_{4}(u),K_{4}(f_{4}(u))).
\end{align*}
Since $f_{i}$ is a bijection for every $i\in\{2,3,4\},$ this implies
\[
\sum_{u\in V(R_{r,s})}\sum_{i=1}^{4}\mathrm{gpn}_{R_{r,s}}(u,K_{i}%
(u))=4\sum_{u\in V(R_{r,s})}\mathrm{gpn}_{R_{r,s}}(u,K_{1}(u)).
\]
Notice that
\[
\mathrm{gpn}_{R_{r,s}}(u,V(R_{r,s}))=\sum_{i=1}^{4}\mathrm{gpn}_{R_{r,s}%
}(u,K_{i}(u))-(r+s-2)-3,
\]
since quadrants $K_{i}(u)$ overlap in the copy of $P_{r}$ and $P_{s}$ which
pass through $u$ (and hence $u$ alone is in all $4$ quadrants). We conclude
that%
\[
\mathrm{gpn}(R_{r,s})=\frac{1}{2}\sum_{u\in V(R_{r,s})}\mathrm{gpn}_{R_{r,s}%
}(u,V(R_{r,s}))+\frac{rs}{2},
\]
where the coefficient $1/2$ before the sum comes from $(u,v)$-geodesics being
included twice in the sum, for $u\not =v,$ and the additional $rs/2$ comes
from $(u,u)$-geodesics being included in the sum only once. Given all the
considerations above, we obtain
\begin{align*}
\mathrm{gpn}(R_{r,s})  &  =\frac{1}{2}\sum_{u\in V(R_{r,s})}\left(  \sum
_{i=1}^{4}\mathrm{gpn}_{R_{r,s}}(u,K_{i}(u))-(r+s+1)\right)  +\frac{1}{2}rs\\
&  =\frac{1}{2}\sum_{u\in V(R_{r,s})}\sum_{i=1}^{4}\mathrm{gpn}_{R_{r,s}%
}(u,K_{i}(u))-\frac{rs}{2}(r+s+1)+\frac{1}{2}rs\\
&  =2\sum_{u\in V(R_{r,s})}\mathrm{gpn}_{R_{r,s}}(u,K_{1}(u))-\frac{rs}%
{2}(r+s)\\
&  =2\sum_{a=1}^{r}\sum_{b=1}^{s}\left(  \binom{a+b}{a}-1\right)  -\frac
{rs}{2}(r+s).
\end{align*}
Since%
\[
\sum_{a=1}^{r}\sum_{b=1}^{s}\binom{a+b}{a}=\binom{r+s+2}{r+1}-s-r-2,
\]
we obtain%
\[
\mathrm{gpn}(R_{r,s})=2\binom{r+s+2}{r+1}-2s-2r-4-2rs-\frac{rs}{2}(r+s)
\]
which yields the claim.
\end{proof}

In the case $r=s$ we get
\[
\mathrm{gpn}(R_{r,s})=2\binom{2r+2}{r+1}-4r-4-r^{2}(r+2).
\]
So
\[
\mathrm{gpn}(R_{r,s})\sim2\binom{2\sqrt{n}+2}{\sqrt{n}+1}\sim\frac{2}%
{\sqrt{\pi\sqrt{n}}}\cdot2^{2\sqrt{n}}\ll\Big(\frac{3}{2}\Big)^{2}\cdot
3^{n/3}.
\]
This means that the geodesic subpath number of $R_{r,s}$ is also significantly
smaller than the geodesic subpath number of $G_{3,n/3}.$

In this section we have shown that the geodesic subpath number of graphs
$G_{3,n/3}$ is quite close to the upper bound from Theorem \ref{thm:bound},
but they are not the same. Hence, in order to establish extremal graphs one
has either to lower the bound of Theorem \ref{thm:bound}, or to find graphs
with larger values of the geodesic path number than $G_{3,n/3}$, or both.

\section{Extremal cacti}

Besides considering extremal graphs among all simple connected graphs on $n$
vertices, one may consider extremal graphs on some particular graph classes.
Trees, which are the usual subclass to be considered, are not particularly
interesting since all trees are geodetic graphs, so Observation
\ref{Observation_geodetic} applies. The next class which is usually considered
are unicyclic graphs or, more broadly, cacti. A \emph{cactus graph} is a graph
in which all cycles are pairwise edge disjoint. Let $\mathcal{C}_{n,k}$ be the
class of all cactus graphs on $n$ vertices with $k$ cycles. We wish to
establish extremal cacti in the class $\mathcal{C}_{n,k}.$

Let us first consider the minimal cacti with respect to the subpath number in
the class $\mathcal{C}_{n,k}.$ Notice that any pair of vertices on an odd
length cycle is connected by precisely one shortest path on the cycle. On the
other hand, on the cycle of even length there exist at least two pairs of
vertices which are connected by two shortest paths on the cycle. This easily
yields the following observation.

\begin{observation}
For a cactus graph $G\in\mathcal{C}_{n,k}$ it holds that $\mathrm{gpn}%
(G)\geq\binom{n+1}{2}$ with equality if and only if each cycle in $G$ is of
odd length.
\end{observation}

The above observation resolves the question of minimal cacti in $\mathcal{C}%
_{n,k}$, so next we consider maximal cacti in $\mathcal{C}_{n,k}$ with respect
to the geodesic subpath number. First, in the next lemma we show that maximal
cacti from $\mathcal{C}_{n,k}$ do not contain an odd length cycle of length
$\geq5.$

\begin{lemma}
\label{Lemma_odd}Let $G\in\mathcal{C}_{n,k}$ be a cactus graph. If $G$
contains a cycle of odd length at least $5,$ then there exists $G^{\prime}%
\in\mathcal{C}_{n,k}$ with $\mathrm{gpn}(G^{\prime})>\mathrm{gpn}(G).$
\end{lemma}

\begin{proof}
Assume that $G\in\mathcal{C}_{n,k}$ contains a cycle $C=v_{1}\cdots v_{g}%
v_{1}$ of length $g,$ where $g\geq5$ is odd. Let $G^{\prime}$ be a cactus
graph obtained from $G$ by removing the edge $v_{1}v_{g}$ and inserting the
edge $v_{2}v_{g}$ instead. Notice that $G^{\prime}$ contains the cycle
$C^{\prime}=v_{2}\cdots v_{g}v_{2}$ instead of $C$ and the edge $v_{1}v_{2}$
becomes a bridge in $G^{\prime}$. Denote by $G_{i}$ the connected component of
$G\backslash E(C)$ which contains $v_{i}.$ The length of the cycle $C^{\prime
}$ in $G^{\prime}$ is $g-1,$ which is even since $g$ is odd. This implies that
for $2\leq i<j\leq g,$ it holds that $\mathrm{gpn}_{G}(u,v)<\mathrm{gpn}%
_{G^{\prime}}(u,v)$ if $u\in V(G_{i})$ and $v\in V(G_{j})$ and $j-i=(g-1)/2.$
Also, $\mathrm{gpn}_{G}(u,v)<\mathrm{gpn}_{G^{\prime}}(u,v)$ when $u\in
V(G_{1})$ and $v\in V(G_{(g+3)/2})$. In all the remaining cases it holds that
$\mathrm{gpn}_{G}(u,v)=\mathrm{gpn}_{G^{\prime}}(u,v).$ Hence, the value of
$\mathrm{gpn}(u,v)$ either increases from $G$ to $G^{\prime}$ or remains the
same. Since there exists at least one pair of vertices for which
$\mathrm{gpn}(u,v)$ increases from $G$ to $G^{\prime},$ that being the pair
$v_{2}$ and $v_{(g+3)/2},$ it follows that $\mathrm{gpn}(G^{\prime
})>\mathrm{gpn}(G)$ as claimed.
\end{proof}

After we have addressed the odd length cycles in the previous lemma, we next
wish to consider the even length cycles. For that purpose, we need the notion
of an antipodal pair of vertices. An \emph{antipodal pair} of vertices on an
even length cycle is a pair of two vertices of the cycle which are
end-vertices of two distinct subpaths of the cycle of equal length. Now, a
cycle of length $k$ is also calles a $k$\emph{-cycle}. In particular,
$3$-cycle is called a \emph{triangle} and $4$-cycle a \emph{square}. A cactus
graph is \emph{girth restricted} if every cycle of $G$ is either a triangle or
a square.

\begin{figure}[h]
\begin{center}%
\begin{tabular}
[t]{ll}%
a) & \raisebox{-0.9\height}{\includegraphics[scale=0.5]{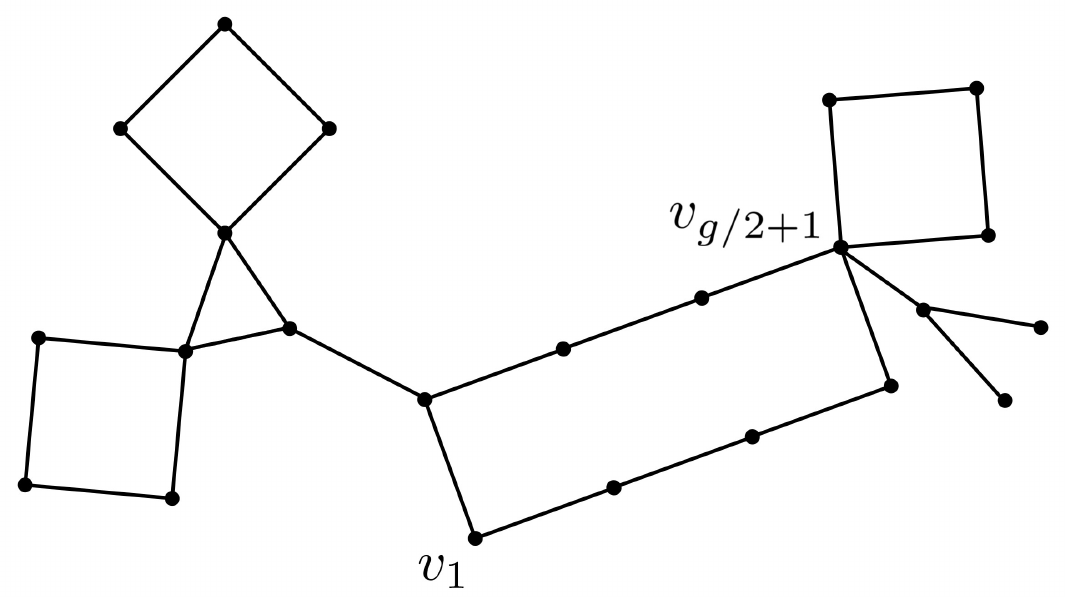}}\\
b) & \raisebox{-0.9\height}{\includegraphics[scale=0.5]{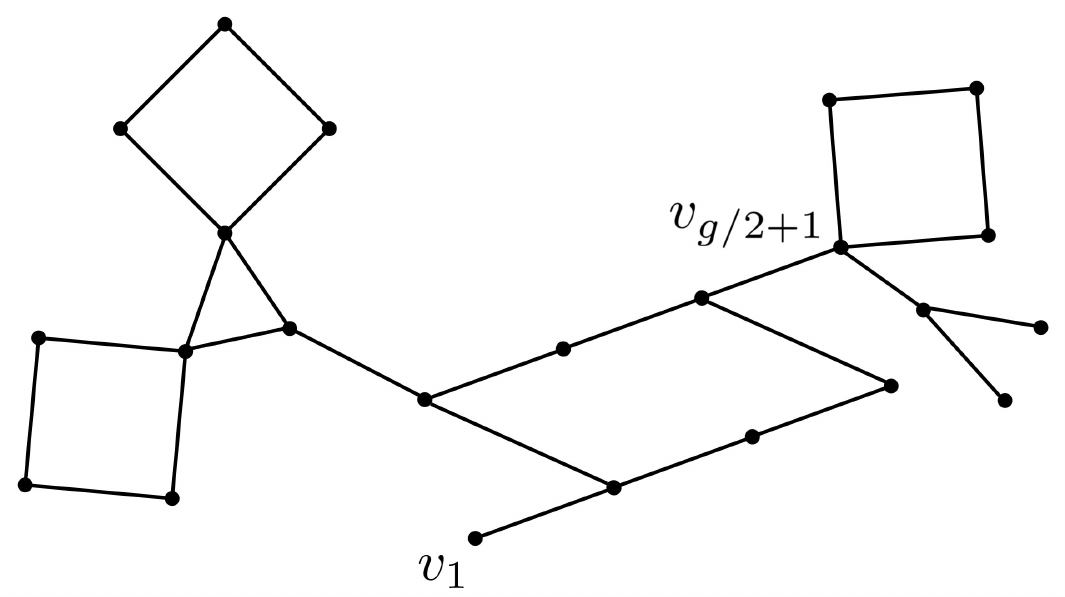}}
\end{tabular}
\end{center}
\caption{The figure shows: a) a cactus graph which is not girth restricted
since it contains a cycle of length $6,$ b) a girth restricted graph
$G^{\prime}$ obtained from $G$ by applying the transformation from the proof
of Lemma \ref{Lemma_girthRestricted}.}%
\label{Fig_girthRestricted}%
\end{figure}

\begin{lemma}
\label{Lemma_girthRestricted}Let $G\in\mathcal{C}_{n,k}$ be a cactus graph. If
$G$ is not girth restricted, then there exists a girth restricted $G^{\prime
}\in\mathcal{C}_{n,k}$ with $\mathrm{gpn}(G^{\prime})>\mathrm{gpn}(G).$
\end{lemma}

\begin{proof}
Due to Lemma \ref{Lemma_odd}, we may assume that all cycles of length at least
$5$ in $G$ are of even length. Assume that $G$ is not girth restricted. This
means that $G$ contains a cycle $C=v_{1}v_{2}\cdots v_{g}v_{1}$ where $g\geq5$
is even. Let $G^{\prime}=G-v_{1}v_{g}+v_{2}v_{g}-v_{g/2}v_{g/2+1}%
+v_{g/2}v_{g/2+2}$, and this transformation of $G$ into $G^{\prime}$ is
illustrated by Figure \ref{Fig_girthRestricted}. Notice that $G^{\prime}$
contains all cycles of $G$ except for the cycle $C.$ Instead of $C$, which is
of length $g,$ the graph $G^{\prime}$ contains the cycle $C^{\prime}%
=v_{2}\cdots v_{g/2}v_{g/2+2}\cdots v_{g}v_{2}$, which is of length $g-2.$ We
wish to show that $\mathrm{gpn}(G^{\prime})>\mathrm{gpn}(G).$

Let $G_{i}$ denote the connected component of $G-E(C)$ which contains $v_{i}.$
Let $u,v\in V(G)$ be a pair of vertices in $G.$ If both $u$ and $v$ belong to
the same component $G_{i},$ then $\mathrm{gpn}_{G^{\prime}}(u,v)=\mathrm{gpn}%
_{G}(u,v).$ So, let us assume that $u$ and $v$ belong to two distinct
connected components, i.e. $u\in V(G_{i})$ and $v\in V(G_{j})$ where we may
assume $i<j.$ Let us first consider the case of $(i,j)\in
\{(1,g/2+2),(2,g/2+1)\}.$ Notice that $\mathrm{gpn}_{G}(v_{i},v_{j})=1$ and
$\mathrm{gpn}_{G^{\prime}}(v_{i},v_{j})=2.$ Due to this, we have
$\mathrm{gpn}_{G^{\prime}}(u,v)>\mathrm{gpn}_{G}(u,v).$ Let us next consider
the case of $(i,j)=(1,g/2+1).$ Since $\mathrm{gpn}_{G}(v_{i},v_{j}%
)=\mathrm{gpn}_{G}(v_{i},v_{j})=2,$ it follows that $\mathrm{gpn}_{G^{\prime}%
}(u,v)=\mathrm{gpn}_{G}(u,v).$ Let us finally consider all the remaining cases
of $(i,j).$ Notice that $v_{i}$ and $v_{j}$ is an antipodal pair on $C$ in $G$
if and only if it is an antipodal pair on $C^{\prime}$ in $G^{\prime}.$ This
implies $\mathrm{gpn}_{G}(v_{i},v_{j})=\mathrm{gpn}_{G}(v_{i},v_{j})$, and
consequently $\mathrm{gpn}_{G^{\prime}}(u,v)=\mathrm{gpn}_{G}(u,v).$

We conclude that $\mathrm{gpn}(u,v)$ either increases from $G$ to $G^{\prime}$
or it remains the same. Since there exists at least one pair of vertices for
which $\mathrm{gpn}(u,v)$ increases from $G$ to $G^{\prime},$ namely $v_{1}$
and $v_{g/2+2},$ we obtain $\mathrm{gpn}(G^{\prime})>\mathrm{gpn}(G)$ as claimed.
\end{proof}

Next, a vertex of a cactus graph is \emph{cyclic} if it belongs to a cycle of
$G.$ A cyclic vertex of $G$ is \emph{active} if $\deg_{G}(v)\geq3,$ otherwise
it is \emph{passive}. A square i.e. a $4$-cycle of $G$ is \emph{multiactive}
if it contains at least two active vertices. Also, we say that a square is
\emph{antipodal} if it contains precisely two active vertices which form an
antipodal pair. Now, an \emph{antipodal} cactus graph is any girth restricted
cactus graph in which every multiactive square is antipodal.

\begin{lemma}
\label{Lemma_antipodal}Let $G\in\mathcal{C}_{n,k}$ be a cactus graph. If $G$
is not antipodal, then there exists an antipodal cactus graph $G^{\prime}%
\in\mathcal{C}_{n,k}$ with $\mathrm{gpn}(G^{\prime})>\mathrm{gpn}(G).$
\end{lemma}

\begin{proof}
By Lemma \ref{Lemma_girthRestricted}, we may assume that $G$ is girth
restricted. Since $G$ is not antipodal, there exists at least one multiactive
square $C=v_{1}v_{2}v_{3}v_{4}v_{1}$ in $G$ which is not antipodal. By $G_{i}$
we denote the connected component of $G-E(C)$ which contains $v_{i}.$

Since $C$ is not antipodal, at least two neighboring vertices of $C$ are
active, say $v_{2}$ and $v_{3}.$ Let $w_{1},\ldots,w_{k}$ be all the neighbors
of $v_{2}$ which belong to the component $G_{2}.$ Since $v_{2}$ is active, we
know that $k\geq1$. Also, let $z_{1},\ldots,z_{q}$ be all the neighbors of
$v_{4}$ which belong to $G_{4}.$ Notice that it may happen that $q=0,$ i.e.
that $v_{4}$ does not have any neighbors in $G$ besides $v_{1}$ and $v_{3}.$
Let $G^{\prime}$ be the graph obtained from $G$ by doing the following rewiring:

\begin{itemize}
\item removing an edge $w_{i}v_{2}$ and inserting $w_{i}v_{1}$ instead for
every $i\in\{1,\ldots,k\};$

\item removing an edge $z_{i}v_{4}$ and inserting $z_{i}v_{3}$ instead for
every $i\in\{1,\ldots,q\}.$
\end{itemize}

Let us show that the number of shortest paths increases in this
transformation. Let $x$ and $y$ be a pair of vertices from $G.$ If neither $x$
nor $y$ belong to $(V(G_{2})\backslash\{v_{2}\})\cup(V(G_{4})\backslash
\{v_{4}\}),$ then the number of shortest paths connecting $x$ and $y$ remains
the same. Hence, let us assume $x\in V(G_{2})\backslash\{v_{2}\},$ and it is
sufficient to consider only this case since the case $x\in V(G_{4}%
)\backslash\{v_{4}\}$ is symmetric.

If $y\in V(G_{1}),$ then each shortest path $P$ in $G$ between $x$ and $y$
passes through the edge $v_{2}v_{1},$ i.e. its vertex sequence contains the
segment $w_{i}v_{2}v_{1}.$ The corresponding shortest path $P^{\prime}$ in
$G^{\prime}$ is obtained from $P$ by removing $v_{2}$ from the vertex
sequence. Hence, $\mathrm{gpn}_{G^{\prime}}(x,y)=\mathrm{gpn}_{G}(x,y).$

If $y\in V(G_{4})\backslash\{v_{4}\},$ then each shortest path $P$ in $G$
between $x$ and $y$ either contains the segment $v_{2}v_{1}v_{4}$ or the
segment $v_{2}v_{3}v_{4}.$ The corresponding shortest path $P^{\prime}$
between $x$ and $y$ in $G^{\prime}$ is obtained from $P$ by replacing this
segment by $v_{1}v_{2}v_{3}$ or the segment $v_{1}v_{4}v_{3},$ respectively.
We again conclude $\mathrm{gpn}_{G^{\prime}}(x,y)=\mathrm{gpn}_{G}(x,y).$

If $y\in V(G_{3})\backslash\{v_{3}\},$ then each shortest path $P$ between $x$
and $y$ in $G$ passes through the edge $v_{2}v_{3},$ i.e. contains the segment
$w_{i}v_{2}v_{3}.$ For each such path $P,$ the two corresponding shortest
paths $P^{\prime}$ and $P^{\prime\prime}$ can be obtained from $P$ by
replacing the segment $w_{i}v_{2}v_{3}$ with $w_{i}v_{1}v_{2}v_{3}$ and
$w_{i}v_{1}v_{4}v_{3},$ respectively. Hence, $\mathrm{gpn}_{G^{\prime}%
}(x,y)=2\mathrm{gpn}_{G}(x,y)>\mathrm{gpn}_{G}(x,y).$

Finally, considering the case when $y\in\{v_{3},v_{4}\},$ notice that
$\mathrm{gpn}_{G^{\prime}}(x,v_{3})=2\mathrm{gpn}_{G}(x,v_{3})$ and
$\mathrm{gpn}_{G^{\prime}}(x,v_{4})=\mathrm{gpn}_{G}(x,v_{4})/2.$ Also, it is
easy to observe that $\mathrm{gpn}_{G}(x,v_{4})=2\mathrm{gpn}_{G}(x,v_{3}).$
We conclude that
\begin{align*}
\mathrm{gpn}_{G^{\prime}}(x,v_{3})+\mathrm{gpn}_{G^{\prime}}(x,v_{4})  &
=2\mathrm{gpn}_{G}(x,v_{3})+\mathrm{gpn}_{G}(x,v_{4})/2\\
&  =\mathrm{gpn}_{G}(x,v_{4})+\mathrm{gpn}_{G}(x,v_{3}).
\end{align*}

We have now considered all the possible cases and concluded that the number of
shortest pairs between $x$ and $y$ either stays the same or increases from $G$
to $G^{\prime},$ i.e. $\mathrm{gpn}(G^{\prime})\geq\mathrm{gpn}(G).$ Moreover,
since $v_{2}$ and $v_{3}$ are active, there exists at least one pair of
vertices $x\in V(G_{2})\backslash\{v_{2}\}$ and $y\in V(G_{3})\backslash
\{v_{3}\}$ for which $\mathrm{gpn}_{G^{\prime}}(x,y)>\mathrm{gpn}_{G}(x,y).$
We conclude that $\mathrm{gpn}(G^{\prime})>\mathrm{gpn}(G).$

Repeating this transformation on all multiactive squares of $G$ which are not
antipodal yields the claim.
\end{proof}

\begin{figure}[h]
\begin{center}%
\begin{tabular}
[t]{ll}%
a) & \raisebox{-0.9\height}{\includegraphics[scale=0.6]{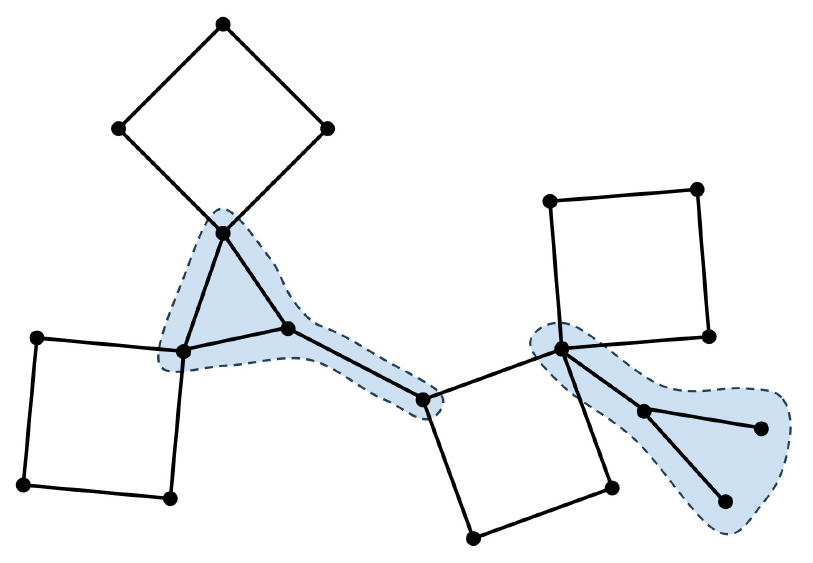}}\\
b) & \raisebox{-0.9\height}{\includegraphics[scale=0.6]{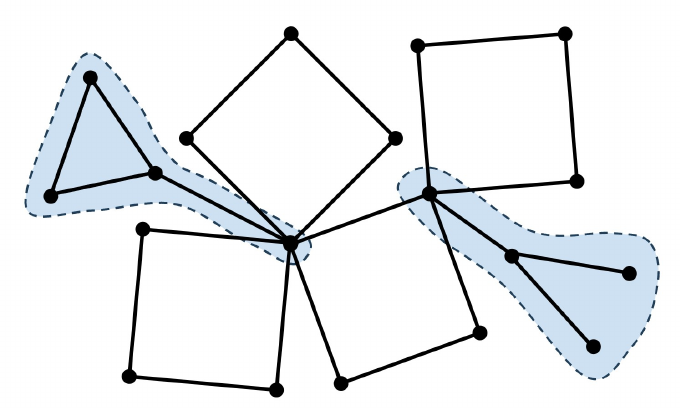}}
\end{tabular}
\end{center}
\caption{The figure shows two cactus graphs and the unipath structures of
these graphs are shaded. Both unipath structures of the graph in b) contain
only one squared vertex, but it belongs to more than one square. The graph a)
contains a unipath structure with more than one squared vertex. None of these
unipath structures is good.}%
\label{Fig_unpathStructures}%
\end{figure}

For an antipodal cactus graph $G\in\mathcal{C}_{n,k},$ let $E_{S}$ denote the
set of all edges of $G$ which belong to the squares. A \emph{unipath
structure} of $G$ is any non-trivial component of $G-E_{S}.$ Notice that for
any pair of vertices which belong to the same unipath structure of $G$, there
is only one shortest path in $G$ connecting them. A unipath structure is
\emph{good} if it contains at most one vertex $v$ which belongs to a square,
and then $v$ belongs to precisely one square. In other words, if we define a
\emph{squared vertex} to be a vertex which belongs to a square, then a unipath
structure is not good if it contains either at least two squared vertices or a
single squared vertex which belongs to at least two squares. Finally, a
\emph{unipath resolved} graph is an antipodal cactus graph in which every
unipath structure is good. An example of such a graph is given in Figure
\ref{Fig_resolved}.

\begin{figure}[h]
\begin{center}
\raisebox{-0.9\height}{\includegraphics[scale=0.6]{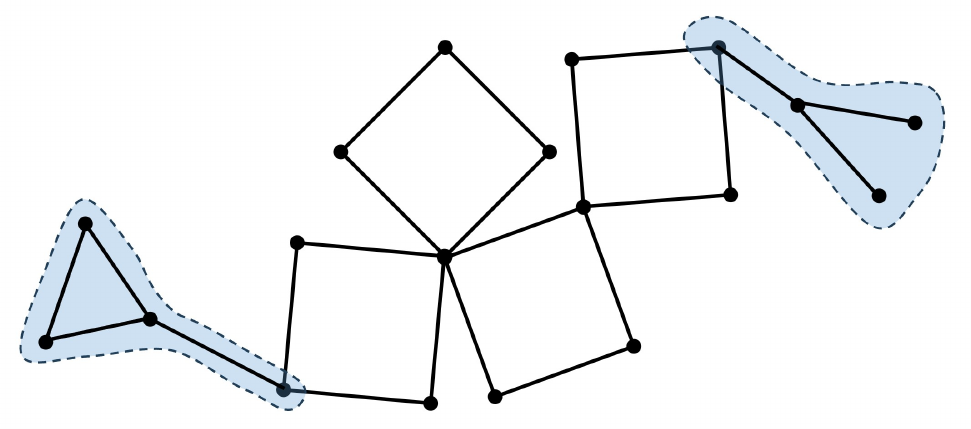}}
\end{center}
\caption{The figure shows a unipath resolved cactus graph. It contains two
unipath structures and both of them are good.}%
\label{Fig_resolved}%
\end{figure}

We will next narrow our search for maximal graphs to unipath resolved cacti,
but before we proceed we need the following notation. For $X,Y\subseteq V(G),$
we define $\mathrm{gpn}_{G}(X,Y)=\sum_{x\in X,y\in Y}\mathrm{gpn}_{G}(x,y).$
If $X=\{x\}$ is a singleton, we will write $\mathrm{gpn}_{G}(x,Y)$ instead of
$\mathrm{gpn}_{G}(\{x\},Y).$ Now, let us establish the following lemma.

\begin{lemma}
\label{Lemma_unipathResolved}Let $G\in\mathcal{C}_{n,k}$ be a cactus graph. If
$G$ is not a unipath resolved graph, then there exists a unipath resolved
graph $G^{\prime}\in\mathcal{C}_{n,k}$ with $\mathrm{gpn}(G^{\prime
})>\mathrm{gpn}(G).$
\end{lemma}

\begin{proof}
Due to Lemma \ref{Lemma_antipodal}, we may assume that $G$ is antipodal. Let
us assume that $G$ is not a unipath resolved. This means that there exists a
unipath structure $H$ in $G$ which is not good. Hence, $H$ contains either at
least two squared vertices or only one squared vertex which belongs to at
least two squares. We consider these two possibilities as two separate cases.

\medskip

\noindent\textbf{Case 1.} $H$\emph{ contains at least two squared vertices.}
An example of such a unipath structure is the left unipath structure of the
graph in Figure \ref{Fig_unpathStructures}.a). Let $S=\{s_{1},\ldots,s_{p}\}$
be the set of squared vertices in $H$, where $p\geq2.$ Denote by $S_{i}$ the
set of squared neighbors of $s_{i}$ which do not belong to $H.$ Notice that
$S_{i}$ contains at least two vertices for every $i\in\{1,\ldots,p\}.$ Let
$G^{\prime}$ be the graph obtained from $G$ by removing edges $s_{i}x$ for
every $x\in S_{i}$ and every $i\in\{2,\ldots,p\},$ and then adding the edge
$s_{1}x$ instead. Notice that $H$ becomes a unipath structure in $G^{\prime}$
which contains only one squared vertex $s_{1}$ which belongs to at least two
squares. Applying this transformation to the graph from Figure
\ref{Fig_unpathStructures}.a) yields the graph from Figure
\ref{Fig_unpathStructures}.b).

Notice that every pair of vertices $x$ and $y$ is connected by the same number
of shortest paths in $G$ and $G^{\prime},$ so $\mathrm{gpn}(G^{\prime
})=\mathrm{gpn}(G).$

\medskip

\noindent\textbf{Case 2.} $H$\emph{ contains a single squared vertex which is
shared by at least two squares.} An example of such a structure are both
unipath structures of the graph in Figure \ref{Fig_unpathStructures}.b). We
denote the vertex set of $H$ by $V_{H}.$ Let $u$ be the unique squared vertex
of $H$ and let $C_{1}$ and $C_{2}$ be two squares which contain $u.$ Notice
that $u$ can be contained in more than two squares, but any two squares are
sufficient. Denote by $G_{i}$ a connected component of $G-u$ which contains a
vertex of $C_{i}$ distinct from $u,$ and let $V_{i}=V(G_{i})$. Without loss of
generality, we may assume $\mathrm{gpn}_{G}(u,V_{1})\leq\mathrm{gpn}%
_{G}(u,V_{2}).$ Let $z$ be a squared vertex of $G_{1}$ furthest from $u.$
Notice that $z$ belongs to only one square $C$ and $z$ is antipodal to an
active vertex of $C.$ Also, let $S_{u}$ be the set of all neighbors of $u$
contained in $H.$ Let $G^{\prime}$ be the graph obtained from $G$ by removing
all edges $ux,$ where $x\in S_{u},$ and adding edges $zx$ instead. Applying
this transformation to both unipath structures of the graph in Figure
\ref{Fig_unpathStructures}.b) yields the graph in Figure \ref{Fig_resolved}.

Since $z$ belongs to only one square, $H$ becomes a good unipath structure of
$G^{\prime}.$ Also, since $z$ is antipodal to an active vertex of $C,$ we know
that $G^{\prime}$ remains antipodal. It remains to show that $\mathrm{gpn}%
(G^{\prime})>\mathrm{gpn}(G).$ Let $P$ be a shortest path between $u$ and $z$
in $G.$ Denote by $c$ the number of distinct squares which share a non-active
vertex with $P$. Notice that $c\geq1,$ since at least $C_{1}$ shares a
non-active vertex with $P.$ Also, notice that $\mathrm{gpn}(x,y)$ may decrease
from $G$ to $G^{\prime}$ only if $x\in V_{H}$ and $y\in V_{1}$, in which case
$\mathrm{gpn}_{G^{\prime}}(x,y)\geq2^{-c}\mathrm{gpn}_{G}(x,y).$ On the other
hand, for $x\in V_{H}$ and $y\in V_{2}$ we have $\mathrm{gpn}_{G^{\prime}%
}(x,y)\geq2^{c}\mathrm{gpn}_{G}(x,y).$

We obtain%
\begin{align*}
\mathrm{gpn}(G^{\prime})-\mathrm{gpn}(G)\geq{}  &  \mathrm{gpn}_{G^{\prime}%
}(V_{H},V_{1})-\mathrm{gpn}_{G}(V_{H},V_{1})\\
&  +\mathrm{gpn}_{G^{\prime}}(V_{H},V_{2})-\mathrm{gpn}_{G}(V_{H},V_{2}).
\end{align*}
Based on these considerations, we conclude%
\[
\mathrm{gpn}(G^{\prime})-\mathrm{gpn}(G)\geq(2^{-c}-1)\mathrm{gpn}_{G}%
(V_{H},V_{1})+(2^{c}-1)\mathrm{gpn}_{G}(V_{H},V_{2}).
\]
Since $c\geq1$ and $\mathrm{gpn}_{G}(V_{H},V_{1})>0,$ we have
\begin{align*}
\mathrm{gpn}(G^{\prime})-\mathrm{gpn}(G)  &  >\mathrm{gpn}_{G}(V_{H}%
,V_{2})-\mathrm{gpn}_{G}(V_{H},V_{1})\\
&  =\left\vert V_{H}\right\vert (\mathrm{gpn}_{G}(u,V_{2})-\mathrm{gpn}%
_{G}(u,V_{1}))\geq0,
\end{align*}
which proves the claim in this case.

\medskip

Since $G$ contains at least one non-good unipath structure, and the unipath
structure of Case 1 reduces to that of Case 2 where the strict inequality
$\mathrm{gpn}(G^{\prime})-\mathrm{gpn}(G)>0$ holds, we have established
$\mathrm{gpn}(G^{\prime})>\mathrm{gpn}(G)$. Since the described transformation
of $G$ to $G^{\prime}$ reduces the number of unipath structures which are not
good, repeated application of the transformation results in a unipath resolved
cactus graph $G^{\prime}$ with $\mathrm{gpn}(G^{\prime})-\mathrm{gpn}(G)>0$.
\end{proof}

\begin{figure}[h]
\begin{center}
\raisebox{-0.9\height}{\includegraphics[scale=0.6]{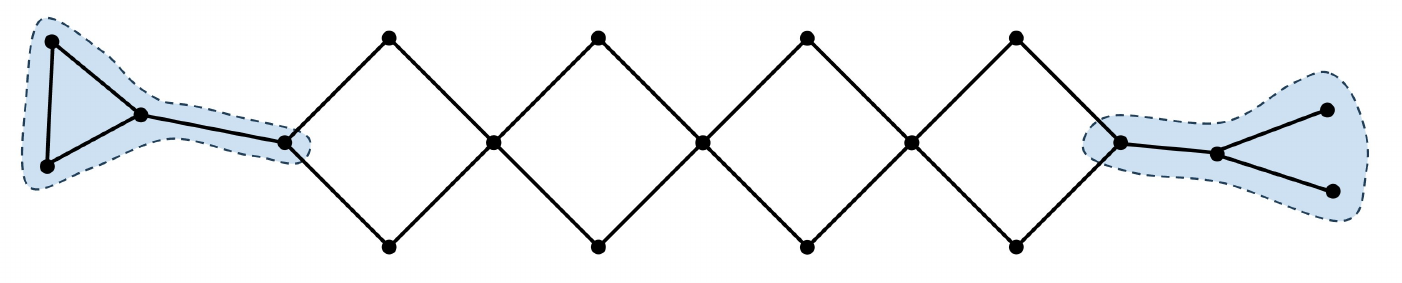}}
\end{center}
\caption{The figure shows a squared chain.}%
\label{Fig_squaredChain}%
\end{figure}

A vertex of $G\in\mathcal{C}_{n,k}$ is \emph{multisquared} (resp.
\emph{bisquared}) if it belongs to at least two (resp. precisely two) squares.
A \emph{squared chain} is a unipath resolved cactus graph in which every
multisquared vertex is bisquared. An example of a squared chain is shown in
Figure \ref{Fig_squaredChain}.

\begin{lemma}
\label{Lemma_squaredChain}Let $G\in\mathcal{C}_{n,k}$ be a cactus graph. If
$G$ is not a squared chain, then there exists a squared chain $G^{\prime}%
\in\mathcal{C}_{n,k}$ with $\mathrm{gpn}(G^{\prime})>\mathrm{gpn}(G).$
\end{lemma}

\begin{proof}
The proof is similar to the proof of Case 2 in Lemma
\ref{Lemma_unipathResolved}. Namely, assume that $G$ is not a squared chain.
Due to Lemma \ref{Lemma_unipathResolved} we may assume that $G$ is unipath
resolved, so $G$ must contain a multisquared vertex $u$ which is not
bisquared. This means $u$ belongs to at least three squares $C_{1},$ $C_{2}$
and $C_{3}.$ For illustration, consider the graph shown in
Figure~\ref{Fig_resolved}. The transformation introduced in this proof
converts it into the squared chain depicted in Figure~\ref{Fig_squaredChain}.

Again, let $G_{i}$ be the connected component of $G-u$ which contains a vertex
of $C_{i},$ and let $V_{i}=V(G_{i}).$ We may assume that $\mathrm{gpn}%
_{G}(u,V_{1})\leq\mathrm{gpn}_{G}(u,V_{2}).$ Let $S_{3}$ denote all neighbors
of $u$ in $G_{3}.$ Notice that $S_{3}$ contains precisely two vertices, and
these two vertices belong also to $C_{3}.$ Denote by $z$ a squared vertex in
$G_{1}$ furthest from $u.$ Let $G^{\prime}$ be the graph obtained from $G$ by
removing edges $ux$ for every $x\in S_{3}$ and inserting edges $zx$ instead.
We wish to show that $\mathrm{gpn}(G^{\prime})>\mathrm{gpn}(G).$

Let $P$ be a shortest path connecting $u$ and $z,$ and $c$ the number of
squares of $G$ which share a non-active vertex with $P.$ Due to $C_{1},$ it is
obvious that $c\geq1.$ Notice that $\mathrm{gpn}(x,y)$ may decrease from $G$
to $G^{\prime}$ only if $x\in V_{3}$ and $y\in V_{1},$ but then
\[
\mathrm{gpn}_{G^{\prime}}(V_{1},V_{3})-\mathrm{gpn}_{G}(V_{1},V_{3}%
)\geq(2^{-c}-1)\mathrm{gpn}_{G}(V_{1},V_{3}).
\]
On the other hand, for $x\in V_{3}$ and $y\in V_{2}$ we have%
\[
\mathrm{gpn}_{G^{\prime}}(V_{2},V_{3})-\mathrm{gpn}_{G}(V_{2},V_{3})\geq
(2^{c}-1)\mathrm{gpn}_{G}(V_{2},V_{3}).
\]
Since $2^{-c}-1>-1$, we conclude that%
\[
\mathrm{gpn}(G^{\prime})-\mathrm{gpn}(G)>\mathrm{gpn}_{G}(V_{2},V_{3}%
)-\mathrm{gpn}_{G}(V_{1},V_{3}).
\]
Notice that for $i\in\{1,2\}$ we have
\begin{align*}
\mathrm{gpn}_{G}(V_{3},V_{i})  &  =\sum_{x\in V_{3}}\mathrm{gpn}_{G}%
(x,V_{i})=\sum_{x\in V_{3}}(\mathrm{gpn}_{G}(x,u)\cdot\mathrm{gpn}_{G}%
(u,V_{i}))\\
&  =\mathrm{gpn}_{G}(u,V_{i})\sum_{x\in V_{3}}\mathrm{gpn}_{G}%
(x,u)=\mathrm{gpn}_{G}(u,V_{i})\mathrm{gpn}_{G}(u,V_{3})
\end{align*}
This further implies%
\[
\mathrm{gpn}(G^{\prime})-\mathrm{gpn}(G)>\mathrm{gpn}_{G}(u,V_{3}%
)(\mathrm{gpn}_{G}(u,V_{2})-\mathrm{gpn}_{G}(u,V_{1}))\geq0,
\]
which establishes that $\mathrm{gpn}(G^{\prime})>\mathrm{gpn}(G).$

Notice that repeated application of the described transformation yields an
antipodal cactus graph $G^{\prime}$ in which every multisquared vertex is
bisquared, but which is not necessarily unipath resolved. But applying Lemma
\ref{Lemma_unipathResolved} to the obtained $G^{\prime}$ yields the claim.
\end{proof}

So far we have narrowed our search for maximal cactus graphs from
$\mathcal{C}_{n,k}$ to squared chains. Let us now consider the maximum
possible number of squares in a cactus graph $G\in\mathcal{C}_{n,k}.$ In order
for $G$ to have $k$ cycles, it must hold that $n\geq2k+1.$ Notice that
$n=2k+1$ implies that all cycles of $G$ are triangles. This means every pair
of vertices in $G$ is connected by precisely one shortest path, so all such
cacti have the same geodesic subpath number. Hence, we will assume $n>2k+1,$
i.e. that $G$ can contain squares.

If all $k$ cycles of $G$ are squares, then $n\geq3k+1.$ This implies that for
$2k+1<$ $n<3k+1,$ the graph $G$ must contain at least one square, but cannot
contain $k$ squares. In order to establish the maximum possible number of
squares in this case, denote by $s$ and $t$ the number of squares and
triangles, respectively. Notice that $n\geq1+2t+3s$ and $t+s=k,$ so we obtain
$3s\leq n-1-2(k-s)$ which implies $s\leq n-2k-1.$

A square chain $G\in\mathcal{C}_{n,k}$ is \emph{maximal} if it contains the
maximum possible number of squares. From the above considerations, this means
$k$ squares for $n\geq3k+1$ and $n-2k-1$ squares for $n<3k+1.$ Notice that a
maximal square chain on $n<3k+1$ vertices cannot contain a bridge.

\begin{lemma}
\label{Lemma_maximal}Let $G\in\mathcal{C}_{n,k}$ be a cactus graph. If $G$ is
not a maximal square chain, then there exists a maximal square chain
$G^{\prime}\in\mathcal{C}_{n,k}$ with $\mathrm{gpn}(G^{\prime})>\mathrm{gpn}%
(G).$
\end{lemma}

\begin{proof}
Assume that $G$ is not a maximal square chain. Due to Lemma
\ref{Lemma_squaredChain} we may assume that $G$ is a squared chain, hence $G$
is not maximal. This implies that at least one out of $k$ cycles in $G$ is a
triangle. Moreover, this implies that $G$ contains a bridge. Let us show that
there exists a squared chain $G^{\prime}$ with more squares than $G$ and
$\mathrm{gpn}(G^{\prime})>\mathrm{gpn}(G).$ We distinguish the following two
cases with respect to whether there exists a bridge in $G$ with one end-vertex
belonging to a triangle.

\medskip

\noindent\textbf{Case 1.} $G$\emph{ contains a bridge with one end-vertex
belonging to a triangle }$C.$ Let us assume that $C=uvwu,$ and that $uz$ is a
bridge which shares the end-vertex $u$ with $C.$ Let $G^{\prime}$ be the graph
obtained from $G$ by removing the edge $uw$ and adding the edge $wz$ instead.
By this transformation, the triangle $C$ of $G$ becomes a square in
$G^{\prime}.$

Notice that $\mathrm{gpn}(x,y)$ does not decrease from $G$ to $G^{\prime}$ for
any pair of vertices $x,y\in V(G).$ On the other hand, $\mathrm{gpn}%
_{G^{\prime}}(u,w)=2>1=\mathrm{gpn}_{G}(u,w).$ Hence, we have established
$\mathrm{gpn}(G^{\prime})>\mathrm{gpn}(G).$ Notice that $G^{\prime}$ may
contain a unipath structure which is not good, but applying Lemma
\ref{Lemma_unipathResolved} to $G^{\prime}$ yields a squared chain with more
squares than $G$ and $\mathrm{gpn}(G^{\prime})>\mathrm{gpn}(G).$

\medskip

\noindent\textbf{Case 2.} $G$\emph{ does not contain a bridge with one
end-vertex belonging to a triangle}. Denote by $E_{S}$ the set of all edges of
$G$ which belong to squares. Also, let $v_{1}$ (resp. $v_{2}$) denote a vertex
which belongs to the first (resp. last) square in the square sequence of $G$,
which is antipodal to the bisquared vertex of the respective square. Denote by
$K_{i}$ the component of $G-E_{S}$ which contains $v_{i},$ for $i\in\{1,2\}.$
Notice that this case is possible only if one of the components, say $K_{1},$
does not contain cycles, and $K_{2}$ does not contain a bridge. Since $G$
contains at least one triangle, this implies $K_{2}$ contains at least one triangle.

\medskip

\noindent\textbf{Case 2.a.} $K_{2}$\emph{ contains precisely one triangle
}$C=v_{2}uvv_{2}.$ Since $G$ contains a bridge, this means $K_{1}$ contains a
bridge, i.e. there exists at least one leaf $w$ in $G$ which is contained in
$K_{1}.$ Denote by $z$ the only neighbor of $w$ in $K_{1}$ and notice that it
may happen that $z=v_{1}.$ Let $G^{\prime}$ be the graph obtained from $G$ by
removing edges $uv$, $wz$ and adding edges $uw$, $vw$ instead, in order to
obtain a new square. We show that $\mathrm{gpn}(G^{\prime})>\mathrm{gpn}(G).$

If $G$ does not contain a square, then the number of geodesics does not
decrease from $G$ to $G^{\prime}$ for any pair of vertices. Yet, by creating a
square $C^{\prime}=v_{2}uwvv_{2}$ in $G^{\prime}$ the number of geodesics for
some pairs of vertices, say $v_{2}$ and $w,$ strictly increases. Hence,
$\mathrm{gpn}(G^{\prime})>\mathrm{gpn}(G).$ On the other hand, if $G$ contains
at least one square, let $c$ denote the number of squares in $G$ and $V_{S}$
the set of all squared vertices of $G$. Obviously, $\left\vert V_{S}%
\right\vert \geq4.$ It is easily verified that $\mathrm{gpn}_{G^{\prime}%
}(w,V_{S})=2\mathrm{gpn}_{G}(w,V_{S})$ and%
\[
\mathrm{gpn}_{G}(w,V_{S})=1+%
{\displaystyle\sum_{i=1}^{c}}
2^{i}+2%
{\displaystyle\sum_{i=1}^{c}}
2^{i-1}=2^{c+2}-3.
\]
Further, for $X=V(K_{1})\backslash\{w,v_{1}\}$ it holds that $\mathrm{gpn}%
_{G^{\prime}}(w,X)=2^{c+1}\mathrm{gpn}_{G}(w,X)=2^{c+1}\left\vert X\right\vert
.$ Notice that it may happen $\left\vert X\right\vert =0.$ Finally, for
$Y=V(K_{2})\backslash\{v_{2}\}=\{u,v\}$ it holds that $\mathrm{gpn}%
_{G^{\prime}}(w,Y)=\left\vert Y\right\vert $ and $\mathrm{gpn}_{G}%
(w,Y)=2^{c}\left\vert Y\right\vert .$ Since for all other pairs of vertices of
$G$ the number of geodesics is the same in $G$ and $G^{\prime},$ we conclude
\begin{align*}
\mathrm{gpn}(G^{\prime})-\mathrm{gpn}(G)  &  =\mathrm{gpn}_{G}(w,V_{S}%
)+(2^{c+1}-1)\left\vert X\right\vert +(1-2^{c})\left\vert Y\right\vert \\
&  \geq2^{c+2}-3+(1-2^{c})\cdot2=2^{c+1}-1>0
\end{align*}
as claimed.

\medskip

\noindent\textbf{Case 2.b.} $K_{2}$\emph{ contains at least two triangles.}
Assume first that $\left\vert V(K_{1})\right\vert \geq3.$ This implies that
$K_{1}$ contains a subpath $P$ of length two. Denote $P=uvw$ and let $e$ be an
edge of $K_{2}$ which belongs to a triangle. Consider the graph $G^{\prime}$
obtained from $G$ by removing an edge $e$ and adding the edge $uw$ instead. It
is easily seen that $\mathrm{gpn}(G^{\prime})=\mathrm{gpn}(G).$ This
transformation decreases the number of triangles in $K_{2}$ by one, and
creates a new triangle in $K_{1}.$ Since $K_{2}$ contains at least two
triangles in $G,$ in $G^{\prime}$ the component $K_{2}$ will contain a bridge
incident to a triangle. Hence, this case reduces to Case 1 and the claim holds.

It only remains to consider the case $\left\vert V(K_{1})\right\vert =2,$
since $K_{1}$ must contain at least one bridge. Let $C=uvw$ be a triangle in
$K_{2}$ with only one active vertex, say $u.$ Obviously, such a triangle must
exist in $K_{2}.$ Let $G^{\prime}$ be a graph obtained from $G$ by removing
edges $uv,$ $uw$ and adding edges $v_{1}v,$ $v_{1}w$ instead. Since
$G^{\prime}$ belongs to Case 1 of this proof, it is sufficient to show
$\mathrm{gpn}(G^{\prime})\geq\mathrm{gpn}(G).$ To see this, let $V_{S}$ be the
set of all squared vertices of $G,$ $X=V(K_{1})\backslash\{v_{1}\},$
$Y=V(K_{2})\backslash\{v_{2},v,w\}$ and $Z=\{v,w\}.$ Notice that
\[
\left\vert X\right\vert =\left\vert V(K_{1})\right\vert -1=1<2\leq\left\vert
V(K_{2})\right\vert -3=\left\vert Y\right\vert .
\]
Let $c$ be the number of squares in $G.$ It is easily seen that
\begin{align*}
\mathrm{gpn}_{G^{\prime}}(Z,V_{S})  &  =\mathrm{gpn}_{G}(Z,V_{S}),\\
\mathrm{gpn}_{G^{\prime}}(Z,X)  &  =2^{-c}\mathrm{gpn}_{G}(Z,X),\\
\mathrm{gpn}_{G^{\prime}}(Z,Y)  &  =2^{c}\mathrm{gpn}_{G}(Z,Y),
\end{align*}
and for all other pairs of vertices the number of geodesics is the same in $G$
and $G^{\prime}.$ Hence,%
\begin{align*}
\mathrm{gpn}(G^{\prime})-\mathrm{gpn}(G)  &  =(2^{-c}-1)\mathrm{gpn}%
_{G}(Z,X)+(2^{c}-1)\mathrm{gpn}_{G}(Z,Y)\\
&  =(2^{-c}-1)\cdot2\cdot\left\vert X\right\vert \cdot2^{c}+(2^{c}%
-1)\cdot2\cdot\left\vert Y\right\vert \\
&  =2(\left\vert Y\right\vert -\left\vert X\right\vert )\left(  2^{c}%
-1\right)  >0
\end{align*}
since $\left\vert Y\right\vert >\left\vert X\right\vert $ and $c\geq1.$ We
have now established $\mathrm{gpn}(G^{\prime})\geq\mathrm{gpn}(G),$ so
applying Case 1 to $G^{\prime}$ yields the claim.
\end{proof}

Let $G\in\mathcal{C}_{n,k}$ be a maximal square chain. Denote by $E_{S}$ the
set of all edges of $G$ which belong to squares of $G.$ Then $G-E_{S}$
contains at most two non-trivial components. Denote by $K_{1}$ and $K_{2}$ the
two largest components of $G-E_{S}$, where we may assume $\left\vert
V(K_{1})\right\vert \geq\left\vert V(K_{2})\right\vert .$ Now we define a
balanced square chain. The definition for $n\geq3k+1$ and $n<3k+1$ is slightly
different, since in the second case the extremal graph contains only squares
and triangles but not bridges. If $n\geq3k+1$ (resp. $2k+1<n<3k+1$) a
\emph{balanced square chain} $G\in\mathcal{C}_{n,k}$ is a maximal square chain
with $\left\vert V(K_{1})\right\vert -\left\vert V(K_{2})\right\vert \leq1$
(resp. $\left\vert V(K_{1})\right\vert -\left\vert V(K_{2})\right\vert \leq2$).

\begin{theorem}
The maximum value of the geodesic subpath number among all cactus graphs from
$\mathcal{C}_{n,k},$ where $n>2k+1,$ is obtained only by a balanced square chain.
\end{theorem}

\begin{proof}
Assume that $G$ is not a balanced square chain. Due to Lemma
\ref{Lemma_maximal} we may assume that $G$ is a maximal square chain, hence
$G$ is not balanced. Since $n>2k+1,$ the graph $G$ must contain at least one
square. We show that there exists a balanced square chain $G^{\prime}$ with
$\mathrm{gpn}(G^{\prime})>\mathrm{gpn}(G).$ We distinguish the following two
cases with respect to $n.$

\medskip

\noindent\textbf{Case 1.} $n\geq3k+1.$ In this case all $k$ cycles of $G$ are
squares. This implies that $K_{1}$ and $K_{2}$ are acyclic structures. Since
$G$ is not balanced, we know that $\left\vert V(K_{1})\right\vert
\geq\left\vert V(K_{2})\right\vert +2.$ Let $u$ be a leaf in $K_{1}$ hanging
at a vertex $v.$ Also, let $w$ be any vertex of $K_{2}.$ The graph $G^{\prime
}$ is obtained from $G$ by removing the edge $uv$ and adding the edge $uw$
instead. We show that $\mathrm{gpn}(G^{\prime})>\mathrm{gpn}(G).$

Denote by $S$ the set of all squared vertices of $G,$ and $V_{i}%
=V(K_{i})\backslash S$ for $i\in\{1,2\}.$ Notice that $V(G)=S\cup V_{1}\cup
V_{2}.$ Let $x,y\in V(G)\backslash\{u\}$ be any pair of vertices, then
obviously, $\mathrm{gpn}_{G^{\prime}}(x,y)=\mathrm{gpn}_{G}(x,y)$. So, let us
assume $x=u.$ It is easily seen that $\mathrm{gpn}_{G^{\prime}}%
(u,S)=\mathrm{gpn}_{G}(u,S),$ even though it may happen $\mathrm{gpn}%
_{G^{\prime}}(u,x)\not =\mathrm{gpn}_{G}(u,x)$ for most vertices $x\in S.$
Also, we have%
\begin{align*}
\mathrm{gpn}_{G^{\prime}}(u,V_{1})  &  =\mathrm{gpn}_{G^{\prime}}%
(u,V_{1}\backslash\{u\})+1=2^{k}\mathrm{gpn}_{G}(u,V_{1}\backslash
\{u\})+1=2^{k}\mathrm{gpn}_{G}(u,V_{1})-2^{k}+1,\\
\mathrm{gpn}_{G^{\prime}}(u,V_{2})  &  =2^{-k}\mathrm{gpn}_{G}(u,V_{2}).
\end{align*}
This implies
\begin{align*}
\mathrm{gpn}(G^{\prime})-\mathrm{gpn}(G)  &  =(2^{k}-1)\mathrm{gpn}%
_{G}(u,V_{1})-2^{k}+1+(2^{-k}-1)\mathrm{gpn}_{G}(u,V_{2})\\
&  =(2^{k}-1)\left\vert V_{1}\right\vert -2^{k}+1+(2^{-k}-1)2^{k}\left\vert
V_{2}\right\vert \\
&  =(2^{k}-1)(\left\vert V_{1}\right\vert -\left\vert V_{2}\right\vert -1)\\
&  \geq(2^{k}-1)(\left\vert V_{2}\right\vert +2-\left\vert V_{2}\right\vert
-1)>0,
\end{align*}
so the claim is established in this case.

\medskip

\noindent\textbf{Case 2.} $2k+1<n<3k+1.$ In this case not all $k$ cycles of
$G$ are squares. Denote by $c$ the number of squares in $G.$ Since $G$ is
maximal square chain and $n>2k+1,$ we know that $c\geq1.$ Recall that a
maximal square chain in this case cannot contain a bridge. Also, we assumed
$G$ is not balanced. In this case that means $\left\vert V(K_{1})\right\vert
\geq\left\vert V(K_{2})\right\vert +3,$ i.e. the unipath component $K_{1}$
contains at least two triangles more than $K_{2}.$ Let $C=uvwu$ be a triangle
of $K_{1}$ on which only $u$ is active. Since $K_{1}$ is a good unipath
component, such cycle $C$ must exist. Denote by $z$ any vertex of $K_{2}.$ Let
$G^{\prime}$ be the graph obtained from $G$ by removing edges $uv$, $uw$ and
adding edges $zv,$ $zw$ instead. We show that $\mathrm{gpn}(G^{\prime
})>\mathrm{gpn}(G).$

We define $S,$ $V_{1}$ and $V_{2}$ as in Case 1, and let $T=\{v,w\}.$ In order
to establish the difference $\mathrm{gpn}(G^{\prime})-\mathrm{gpn}(G),$ we
have to consider only $\mathrm{gpn}_{G^{\prime}}(x,y)-\mathrm{gpn}_{G}(x,y)$
for $x\in T$ and $y\in V(G)\backslash T,$ since for all other pairs of
vertices the number of shortest paths is the same in $G^{\prime}$ as in $G$.
Notice that%
\begin{align*}
\mathrm{gpn}_{G^{\prime}}(T,V_{1}\backslash T)  &  =2^{c}\mathrm{gpn}%
_{G}(T,V_{1}\backslash T),\\
\mathrm{gpn}_{G^{\prime}}(T,V_{2})  &  =2^{-c}\mathrm{gpn}_{G}(T,V_{2}),\\
\mathrm{gpn}_{G^{\prime}}(T,S)  &  =\mathrm{gpn}_{G}(T,S).
\end{align*}
This implies%
\begin{align*}
\mathrm{gpn}(G^{\prime})-\mathrm{gpn}(G)  &  =(2^{c}-1)\mathrm{gpn}%
_{G}(T,V_{1}\backslash T)+(2^{-c}-1)\mathrm{gpn}_{G}(T,V_{2})\\
&  =(2^{c}-1)\cdot2(\left\vert V_{1}\right\vert -2)+(2^{-c}-1)\cdot2\left\vert
V_{2}\right\vert \\
&  \geq(2^{c}-1)\cdot2(\left\vert V_{2}\right\vert +3-2)+(2^{-c}%
-1)\cdot2\left\vert V_{2}\right\vert \\
&  =2^{1-c}\left(  2^{c}-1\right)  \left(  2^{c}\left\vert V_{2}\right\vert
-\left\vert V_{2}\right\vert +2^{c}\right)  >0,
\end{align*}
and the claim is established.
\end{proof}

\begin{figure}[ptbh]
\begin{center}%
\begin{tabular}
[t]{ll}%
\raisebox{-0.9\height}{\includegraphics[scale=0.6]{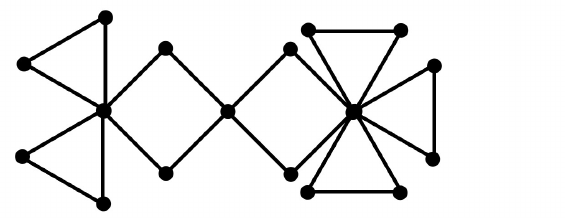}} &
\raisebox{-0.9\height}{\includegraphics[scale=0.6]{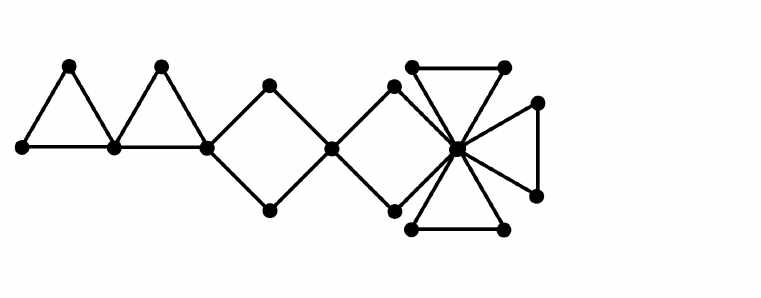}}\\
\raisebox{-0.9\height}{\includegraphics[scale=0.6]{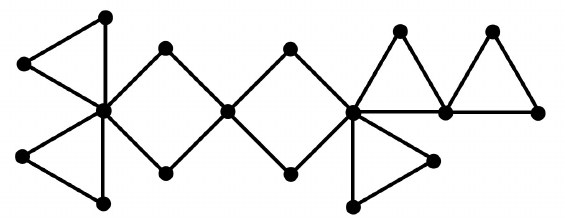}} &
\raisebox{-0.9\height}{\includegraphics[scale=0.6]{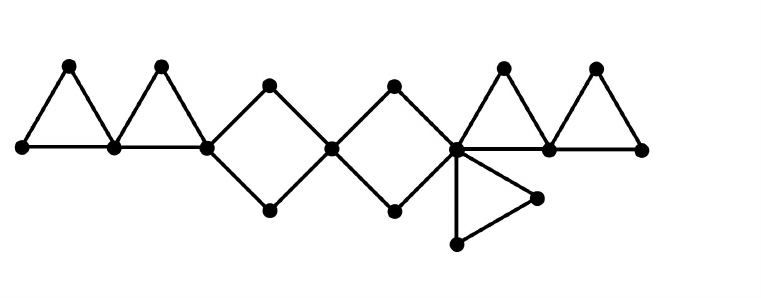}}\\
\raisebox{-0.9\height}{\includegraphics[scale=0.6]{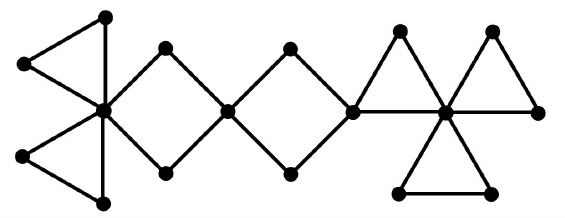}} &
\raisebox{-0.9\height}{\includegraphics[scale=0.6]{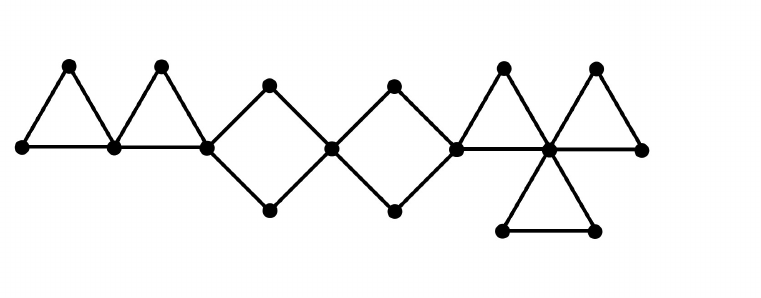}}\\
\raisebox{-0.9\height}{\includegraphics[scale=0.6]{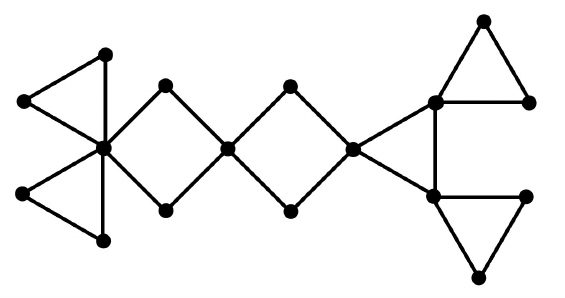}} &
\raisebox{-0.9\height}{\includegraphics[scale=0.6]{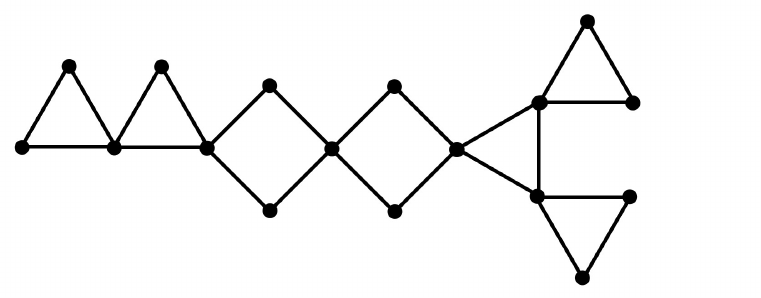}}\\
\raisebox{-0.9\height}{\includegraphics[scale=0.6]{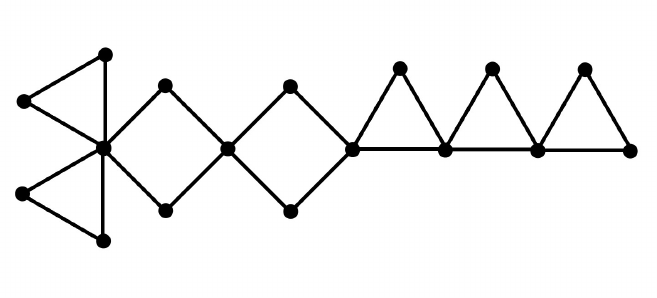}} &
\raisebox{-0.9\height}{\includegraphics[scale=0.6]{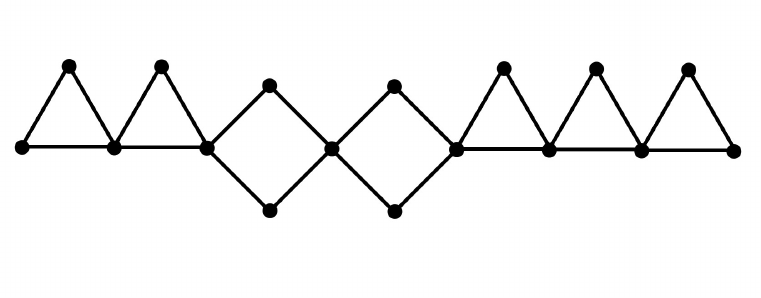}}
\end{tabular}
\end{center}
\caption{All maximal cacti for $n=17$ and $k=7.$}%
\label{Fig_maxTriangles}%
\end{figure}

To recapitulate the results on maximal cacti from this section, we
characterize maximal cacti in the class $\mathcal{C}_{n,k}$ of all cacti on
$n$ vertices with $k$ cycles. This class is non-empty only for $n\geq2k+1.$
The characterization of maximal cacti is distinct for $n<3k+1$ and
$n\geq3k+1.$

In the case $n<3k+1,$ a cactus graph does not contain enough vertices to
contain $k$ squares, so some of the cycles must be triangles. A cactus graph
in this case cannot contain a bridge, since a bridge could be used to enlarge
one triangle into a square. Notice that unipath resolved cacti can contain at
most two non-trivial unipath structures, hanging at the first and the last
square of the square chain. In this case the two unipath structures of maximal
cacti contain only triangles. Also, these two structures must have almost the
same number of vertices. This means that it is either the same if the number
of triangles is even, or it differs by $2$ if the number of triangles is odd.
We illustrate this case by Figure \ref{Fig_maxTriangles} where all maximal
cacti for $n=17$ and $k=7$ are shown.

\begin{figure}[ptbh]
\begin{center}%
\begin{tabular}
[t]{ll}%
\raisebox{-0.9\height}{\includegraphics[scale=0.6]{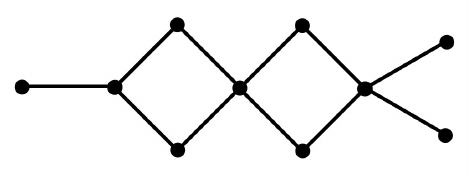}} &
\raisebox{-0.9\height}{\includegraphics[scale=0.6]{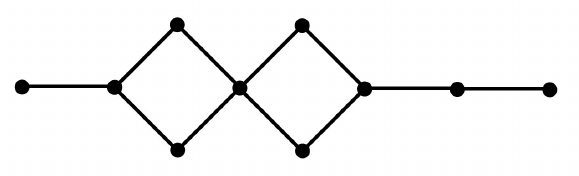}}
\end{tabular}
\end{center}
\caption{All maximal cacti for $n=9$ and $k=2.$}%
\label{Fig_maxAcyclic}%
\end{figure}

In the case of $n\geq3k+1,$ all $k$ cycles of a maximal cactus graph must be
squares. Hence, the two unipath structures must be acyclic. Since their number
of vertices must be almost the same, this implies that it is either the same
or it differs by one, which depends on the parity of $n.$ This case is
illustrated by Figure \ref{Fig_maxAcyclic} which shows all maximal cacti for
$n=9$ and $k=2.$

\bigskip

\section{Concluding remarks}

In this paper we count geodesic subpaths in graphs. Namely, motivated by the
number of subtrees and the number of subpaths which already exist in
literature, but are rather large quantities, we introduce the geodesic subpath
number which is related to them, but is obviously smaller. We initiate the
study of this new invariant by searching for extremal graphs with respect to
the geodesic subpath number among all simple connected graphs on $n$ vertices.
The geodetic graphs, i.e. graphs in which every pair of vertices is connected
by precisely one geodesic path, are minimal.

The question of characterizing maximal graphs seems to be more difficult. In
Theorem \ref{thm:bound} we provide the following upper bound%
\begin{equation}
\mathrm{gpn}(G)\leq\frac{n}{4}(23\cdot3^{\frac{n-6}{3}}+1).\label{For_bound}%
\end{equation}
We next establish the exact value of the geodesic subpath number for several
graph families such as grids, hypercubes and graphs $G_{k,t}$ which are
sequential joins of $t$ copies of an empty graph $D_{k}$ on $k$ vertices.
Among these, we establish that $G_{3,n/3}$ has the largest value of the
geodesic subpath number. Since%
\begin{equation}
\mathrm{gpn}(G_{3,n/3})\sim\left(  \frac{3}{2}\right)  ^{2}\cdot
3^{n/3},\label{For_max}%
\end{equation}
the gap between the established upper bound and the value of $\mathrm{gpn}%
(G_{3,n/3})$ is still significant. Hence, we propose the following problem.

\begin{problem}
Find an upper bound on the geodesic subpath number smaller than \emph{(}%
\ref{For_bound}\emph{)} and/or find a graph $G$ with the value of the geodesic
subpath number larger than \emph{(}\ref{For_max}\emph{)}.
\end{problem}

We believe that there is more space to improve the upper bound of Theorem
\ref{thm:bound} than to find graphs $G$ with $\mathrm{gpn}(G)>\mathrm{gpn}%
(G_{3,n/3}).$ The reason for this is that in the proof of Theorem
\ref{thm:bound} a large number of long geodesics is assumed, but long
geodesics imply many pairs of vertices connected by a shorter geodesic, which
is not adequately tackled here.

In this paper we also characterize extremal graphs in the class of cacti on
$n$ vertices with $k$ cycles. Other classes of graphs might also be
interesting. For example, among complete bipartite graphs it seems easy to
establish that maximal graphs are the graphs with partite sets of the same or
almost the same cardinality. Yet, removing a perfect matching from such a
graph may result in a bipartite graph with a larger value of the geodesic
subpath number. Hence, the following problem might be interesting.

\begin{problem}
Characterize maximal graphs with respect to the geodesic subpath number in the
class of bipartite graphs on $n$ vertices.
\end{problem}

\vskip1pc \noindent\textbf{Acknowledgments.}~~The authors are very grateful to
the reviewer for valuable comments for the improvement of the paper. Also, the
authors acknowledge the partial support by Slovak research grants VEGA
1/0011/25, VEGA 1/0069/23, APVV-23-0076 and APVV-22-0005, by ARIS\ projects
J1-3002 and J1-70016, program\ P1-0383, bilateral Slovenian-Croatian project
BI-HR/25-27-004 and the annual work program of Rudolfovo, by Project
KK.01.1.1.02.0027 co-financed by the European Regional Development Fund, by
the Croatian Science Foundation under project number HRZZ-IP-2024-05-2130, by
Croatian Ministry of Science, Education and Youth trough the bilateral
Croatian-Slovenian project 2025-26, by the NextGeneration EU foundation via
IP-UNIST-17 (GEORAZ), by the National Natural Science Foundation of China
(No.12371354) and the Montenegrin Chinese Science and Technology (No.4-3).




\begin{thebibliography}{99}                                                                                               %


\bibitem {AtajanSpanning}T.~Atajan, X.~Yong, H.~Inaba, An efficient approach
for counting the number of spanning trees in circulant and related graphs,
\emph{Discrete Math.} \textbf{310} (2010) 1210--1221.

\bibitem {BapatSpanning}R.~B.~Bapat, A.~K.~Lal, S.~Pati, Laplacian spectrum of
weakly quasi-threshold graphs, \emph{Graphs Comb.} \textbf{24} (2008) 273--290.

\bibitem {BrownSpanning}T.~J.~N.~Brown, R.~B.~Mallion, P.~Pollak, A.~Roth,
Some methods for counting the spanning trees in labeled molecular graphs,
examined in relation to certain fullerenes, \emph{Discrete Appl. Math.}
\textbf{67} (1996) 51--66.

\bibitem {Cornelsen2022}S.~Cornelsen, M.~Pfister, H.~F\"{o}rster,
M.~Gronemann, M.~Hoffmann, S.~Kobourov, T.~Schneck, Drawing shortest paths in
geodetic graphs, \emph{J. Graph Algorithms Appl.} \textbf{26(3)} (2022) 353--361.

\bibitem {Czabarka2008}\'{E}.~Czabarka, L.~Sz\'{e}kely, S.~Wagner, The inverse
problem for certain tree parameters, \emph{Discrete Appl. Math. }%
\textbf{157(15)} (2009) 3314--3319.

\bibitem {Elder2021}M.~Elder, A.~Piggott, Rewriting systems, plain groups, and
geodetic graphs, \emph{Theor. Comput. Sci.} \textbf{903} (2022) 134-144.

\bibitem {FengShortest}Q.~Feng, Y.~Peng, W.~Zhang, X.~Lin, Y.~Zhang, DSPC:
Efficiently Answering Shortest Path Counting on Dynamic Graphs,
\emph{Proceedings of the 27th International Conference on Extending Database
Technology (EDBT)} (2024) 116--128.

\bibitem {Gorovoy}D.~Gorovoy, D.~Zmiaikou, On graphs with unique geodesics and
antipodes, \emph{Discrete Math.} \textbf{347(4)} (2024) 113864.

\bibitem {Kirk2008}R.~Kirk, H.~Wang, Largest number of subtrees of trees with
a given maximum degree, \emph{SIAM J. Discrete Math.} \textbf{22} (2008) 985--995.

\bibitem {Knor2025}M.~Knor, J.~Sedlar, R.~\v{S}krekovski, Y.~Yang, Invitation
to the subpath number, \emph{Appl. Math. Comput.} \textbf{509} (2026) 129646.

\bibitem {Knor2025b}M.~Knor, J.~Sedlar, R.~\v{S}krekovski, Y.~Yang, The
subpath number of cactus graphs, \emph{Comp. Appl. Math.} \textbf{45} (2026) 102.

\bibitem {risteSurvey1}M.~Knor, R.~\v{S}krekovski, A.~Tepeh, Mathematical
aspects of Wiener index, \emph{Ars Math. Contemp.} \textbf{11} (2016) 327--352.

\bibitem {risteSurvey2}M.~Knor, R.~\v{S}krekovski, A.~Tepeh, Selected topics
on Wiener index, \emph{Ars Math. Contemp.} (2024) \#P4.07.

\bibitem {Kuziak}D.~Kuziak, I.~G.~Yero, Metric dimension related parameters in
graphs: A survey on combinatorial, computational and applied results. (2021)
arXiv preprint arXiv:2107.04877 [math.CO].

\bibitem {Cacti2022}J.~Li, K.~Xu, T.~Zhang, H.~Wang, S.~Wagner, Maximum number
of subtrees in cacti and block graphs, \emph{Aequat. Math.} \textbf{96} (2022) 1027--1040.

\bibitem {NikolopoulosAlgoritamSpanning}S.~D.~Nikolopoulos, L.~Palios,
C.~Papadopoulos, Counting spanning trees using modular decomposition,
\emph{Theor. Comput. Sci.} \textbf{526} (2014) 41--57.

\bibitem {NikopoulosSpanning}S.~D.~Nikolopoulos, P.~Rondogiannis, On the
number of spanning trees of multi-star related graphs, \emph{Inf. Process.
Lett.} \textbf{65} (1998) 183--188.

\bibitem {Ore1962}\O .~Ore, Theory of Graphs, Amer. Math. Soc. Colloq. Publ.,
Vol. 38, Amer. Math. Soc., Providence, RI, 1962.

\bibitem {PengShortest}Y.~Peng, J.~X.~Yu, S.~Wang, PSPC: Efficient Parallel
Shortest Path Counting on Large-Scale Graphs, \emph{Proceedings of the 39th
IEEE International Conference on Data Engineering (ICDE)} (2023) 896--908.

\bibitem {Peterin}I.~Peterin, J.~Sedlar, R.~\v{S}krekovski, I.~G.~Yero,
Resolving vertices of graphs with differences, \emph{Comput. Appl. Math.}
\textbf{43 (5)} (2024) 275.

\bibitem {QiuShortest}Y.~Qiu, R.~Jin, Z.~Zhao, Efficient Shortest Path
Counting on Large Road Networks, \emph{Proceedings of the VLDB Endowment}
\textbf{15(10)} (2022) 2098--2110.

\bibitem {Szekely2005}L.~A.~Sz\'{e}kely, H.~Wang, On subtrees of trees,
\emph{Adv. Appl. Math.} \textbf{34} (2005) 138--155.

\bibitem {Szekely2007}L.~A.~Sz\'{e}kely, H.~Wang, Binary trees with the
largest number of subtrees, \emph{Discrete Appl. Math.} \textbf{155} (2007) 374--385.

\bibitem {Indijka}S. Varghese, A. Lakshmanan S., S. Arumugam, Two extensions
of Leech labeling to the class of all graphs, \emph{AKCE Int. J. Graphs Comb.}
\textbf{19(2)} (2022) 159--165.

\bibitem {Xu2021}K.~Xu, J.~Li, H.~Wang, The number of subtrees in graphs with
given number of cut edges, \emph{Discrete Appl. Math.} \textbf{304} (2021) 283--296.

\bibitem {Yamamoto2017}M.~Yamamoto, Approximately counting paths and cycles in
a graph, \emph{Discrete Appl. Math. }\textbf{217} (2009) 381--387.

\bibitem {YanSpanning}W.-M.~Yan, W.~Myrvold, K.-L.~Chung, A formula for the
number of spanning trees of a multi-star related graph, \emph{Inf. Process.
Lett.} \textbf{68} (1998) 295--298.

\bibitem {Zhang2013}X.~M.~Zhang, X.~D.~Zhang, D.~Gray, H.~Wang, The number of
subtrees of trees with given degree sequence, \emph{J. Graph Theory}
\textbf{73} (2013) 280--295.
\end{thebibliography}
\end{document}